\documentclass{article}
\usepackage{amsmath}
\usepackage{amssymb}
\usepackage[colorlinks=true, bookmarks=true, pdfstartview=FitH, linktocpage=true, linkcolor = magenta, citecolor = blue]{hyperref}

\hypersetup{
colorlinks = true,
linkcolor = magenta,
citecolor = blue,
}
\newtheorem{thm}{Theorem}
 
\newtheorem{proposition}{Proposition}
\newtheorem{cor}{Corollary}
\newtheorem{df}{Definition}
\newtheorem{rem}{Remark\/}

\newenvironment{proof}{\noindent {\bf Proof.} \ignorespaces}{\nopagebreak\hspace*{\fill}$\square$\newline}
\newcommand \capital {\mathbb} 
\newcommand \italic {\mathcal} 
\newcommand \boldletter {\boldsymbol} 

\begin{document}

\title{An improved existence theorem for rigid nonlinearly elastic plates}
\author{Trung Hieu Giang
\footnote{Addresses of Trung Hieu Giang: Charles University, Prague, Czechia.  Institute of Mathematics, Vietnam Academy of Science and Technology, 18 Hoang Quoc Viet, Cau Giay, Ha Noi, Vietnam.} 
\  and
Cristinel Mardare
\footnote{corresponding author}
\footnote{Addresses of Cristinel Mardare: Sorbonne Universit\'e, Universit\'e Paris Cit\'e, CNRS, INRIA, Laboratoire Jacques-Louis Lions, LJLL, 
F-75005 Paris, France. Institute for Advanced Study in Mathematics, Harbin Institute of Technology, China.
}}
\date{\today}
\maketitle

\begin{abstract}
	A plate is rigid if its admissible displacement fields inducing vanishing two-dimensional strain tensors must vanish. We prove that the nonlinear model of Kirchhoff-Love for such a plate has a solution for any applied forces and boundary conditions. Then we give  sufficient conditions on the data ensuring the rigidity of the plate. Together these results substantially generalize an existence theorem by Rabier whereby the plate is assumed to be clamped on its entire boundary. 
\end{abstract}

\medskip

\noindent\textbf{2020 Mathematical Subject Classification:} 74B20, 74G22, 74K20.

\noindent\textbf{Keywords:} {Nonlinear elasticity, Plates, Kirchhoff-Love, Equilibrium problems, Existence theory.}


\maketitle


\section{Introduction}
\label{s1}

Consider a plate with thickness $2\varepsilon>0$ made of a homogeneous and isotropic elastic material with Lam\'e constants $\lambda\geq 0$ and $\mu>0$. Assume that the plate occupies in its natural state (that is, an unstressed state) the three-dimensional set  
\begin{equation*}
	{\overline\omega}\times [-\varepsilon,\varepsilon]
\end{equation*}
where $\omega\subset{\capital{R}}^2$ is a bounded and connected open set whose boundary is Lipschitz-continuous in the sense of Adams \& Fournier \cite{ada}, and that it is subjected to applied body and surface forces with densities $(f^\varepsilon_i)\in L^2(\Omega^\varepsilon;{\capital{R}}^3)$ per unit volume and $(g^\varepsilon_i)\in L^2(\Gamma^\varepsilon_\pm;{\capital{R}}^3)$ per unit area, where
\begin{equation*}\aligned{}
	\Omega^\varepsilon&:=\omega\times (-\varepsilon,\varepsilon),\\
	\Gamma^\varepsilon_\pm&:=\omega\times \{-\varepsilon,\varepsilon\},
	\endaligned\end{equation*}

Then Kirchhoff-Love's two-dimensional theory of plates asserts that the displacement field  
\begin{equation*}
	{\boldletter{u}}:=(u_i):\omega\to{\capital{R}}^3
\end{equation*}
of the middle section of the plate arising in response to these applied forces minimizes the functional $J:V(\omega)\to{\capital{R}}$, where  
\begin{equation*}
	V(\omega)\subset H^1(\omega)\times H^1(\omega)\times H^2(\omega)
\end{equation*}
is the set of all admissible displacement fields of the middle section of the plate, 
\begin{equation}\label{e1}\aligned{} 
	J({\boldletter{u}}):=\frac{\varepsilon}2 \|{\boldletter{E}}({\boldletter{u}})\|_{\lambda,\mu}^2+\frac{\varepsilon^3}6 \|{\boldletter{F}}({\boldletter{u}})\|_{\lambda,\mu}^2-L({\boldletter{u}}) 
	\endaligned\end{equation}
is the two-dimensional total energy of the plate, and
\begin{equation}\label{e2}\aligned{} 
	L({\boldletter{u}}):=\int_{\omega} (p^\varepsilon_i u_i- q^\varepsilon_\alpha\partial_\alpha u_3) dy
	\endaligned\end{equation}
is the work of the two dimensional forces acting on the middle section of the plate.

The tensor fields ${\boldletter{E}}({\boldletter{u}})=(E_{\alpha\beta}({\boldletter{u}}))\in L^2(\omega;{\capital{S}}^2)$ and ${\boldletter{F}}({\boldletter{u}})=(F_{\alpha\beta}({\boldletter{u}}))\in L^2(\omega;{\capital{S}}^2)$ appearing in the definition \eqref{e1} of the functional $J$ are defined by 
\begin{equation}\label{e3}\aligned{} 
	E_{\alpha\beta}({\boldletter{u}})&:=\frac12(\partial_\alpha u_\beta+\partial_\beta u_\alpha+\partial_\alpha u_3\partial_\beta u_3),\\
	F_{\alpha\beta}({\boldletter{u}}) &:=\partial_{\alpha\beta} u_3,
	\endaligned\end{equation}
for all ${\boldletter{u}}=(u_i)\in H^1(\omega)\times H^1(\omega)\times H^2(\omega)$, while the norm $\|\cdot\|_{\lambda,\mu}: L^2(\omega;{\capital{S}}^2)\to{\capital{R}}$ appearing in the same definition is defined by 
\begin{equation}\label{e4}\aligned{} 
	\|(s_{\alpha\beta})\|_{\lambda,\mu}^2:=\int_\omega a^{\alpha\beta\sigma\tau}s_{\sigma\tau}(y)s_{\alpha\beta}(y) dy
	\endaligned\end{equation}
for all $(s_{\alpha\beta})\in L^2(\omega;{\capital{S}}^2)$, where 
\begin{equation}\label{e5}\aligned{} 
	a^{\alpha\beta\sigma\tau}:=\frac{4\lambda\mu}{\lambda+2\mu} \delta_{\sigma\tau}\delta_{\alpha\beta}
	+2\mu( \delta_{\alpha\sigma}\delta_{\beta\tau}+ \delta_{\alpha\tau}\delta_{\beta\sigma}). 
	\endaligned\end{equation}
Note that the four-order tensor $a^{\alpha\beta\sigma\tau}$ is the two-dimensional elasticity tensor associated with the elastic material constituting the plate and that the norm $\|\cdot\|_{\lambda,\mu}$ is equivalent with the usual norm of the Lebesgue space $L^2(\omega;{\capital{S}}^2)$,  thanks to the positivity of the Lam\'e constants of the elastic material constituting the plate. In fact, as is well known, the positivity of the Lam\'e constants can be replaced by the weaker assumption that $\mu>0$ and $\lambda>-\frac23\mu$. 

Finally, the functions $p_i^\varepsilon\in L^2(\omega)$ and $q_\alpha^\varepsilon\in L^2(\omega)$ appearing in the definition \eqref{e2} of the linear form $L$ are defined in terms of the densities of the applied forces on the plate by the relations:
\begin{equation}\label{e6}\aligned{} 
	p^\varepsilon_i&:=\int_{-\varepsilon}^\varepsilon f^\varepsilon_i(\cdot,x_3) dx_3+g^\varepsilon_i(\cdot,+\varepsilon)+g^\varepsilon_i(\cdot,-\varepsilon),\\
	q^\varepsilon_\alpha&:=\int_{-\varepsilon}^\varepsilon x_3 f^\varepsilon_\alpha(\cdot,x_3) dx_3+\varepsilon\Big(g^\varepsilon_\alpha(\cdot,+\varepsilon)-g^\varepsilon_\alpha(\cdot,-\varepsilon)\Big).
	\endaligned\end{equation}

Over the past few decades, several papers have been devoted to studying existence results for this nonlinearly elastic plate model under various sets of assumptions. Ne\v cas \& Naumann \cite{nec} proved that the functional $J$ has at least one minimizer when the plate is subjected to a perpendicular load. Ciarlet \& Destuynder \cite{cd} proved the existence of a minimizer to $J$ provided that the norms $\|p_\alpha^\varepsilon\|_{L^2(\omega)}$ are small enough. Rabier \cite{rab} proved that if the plate is totally clamped, then the smallness assumption in \cite{cd} is not necessary for proving the existence of a minimizer to $J$. 

Among the above results, most relevant to this paper is Rabier's, where boundary conditions are used to drop Ciarlet \& Destuynder's smallness assumption on the tangential forces needed in \cite{cd} to prove the coerciveness of the functional $J$ over the space $V(\omega)$, which in turn implies the existence of minimizers for the above Kirchhoff-Love plate model. But Rabier's method applies only to totally clamped plates, that is, for plates whose admissible displacement fields satisfy the full Dirichlet boundary conditions associated with the Euler-Lagrange equations satisfied by the smooth minimizers of $J$. 

Essential in Rabier's approach is that a totally clamped plate is rigid, in the sense that its admissible displacement fields satisfies property \eqref{e7} in Section \ref{s3} below, which is a necessary condition for the infimum of $J$ over $V(\omega)$ to be finite for every applied force. 

One objective of this paper, achieved in Theorem \ref{t2}, is to prove that this rigidity property is also sufficient for the existence of minimizers of $J$ over $V(\omega)$. 

The other objective of this paper, achieved in Theorems \ref{t3}, \ref{t4}, \ref{t5} and \ref{t6}, is to give sufficient conditions for the rigidity property to hold. These sufficient conditions are boundary conditions on the transverse component $u_3$ of the displacement fields $\boldletter{u}=(u_i)\in V(\omega)$ in Theorem \ref{t3}, while they are boundary conditions on the tangential components $u_\alpha$ of $\boldletter{u}$ in Theorems \ref{t4}, \ref{t5} and \ref{t6}. The existence theorems obtained by combining these rigidity results with Theorem \ref{t2} improve Rabier's existence theorem in \cite{rab}, since they hold under weaker boundary conditions than the full Dirichlet boundary conditions used in \cite{rab}; see Corollaries \ref{c1}, \ref{c2}, \ref{c3} and \ref{c4}. 

The paper is organized as follows. In the next section, we introduce the notation used in this paper. Section \ref{s3} establishes an existence theorem for the Kirchhoff-Love nonlinear plate model under the sole assumption that the plate is rigid (so irrespectively of the conditions ensuring the rigidity of the plate). 

Sections \ref{s4} and \ref{s5} give several sets of sufficient conditions ensuring the rigidity of the plate: Section \ref{s4} under the assumption the transversal displacements vanish on its boundary, while Section \ref{s5} under the assumption that the tangential displacements vanish on a sufficiently large part of this boundary in a specific sense.

\section{Notation}
\label{s2}

Throughout this paper, Greek indices and exponents range in the set $\{1,2\}$, while Latin indices and exponents range in the set $\{1,2,3\}$, unless they are used for indexing sequences or otherwise specified in the text. 

The summation convention with respect to repeated indices and exponents is used, so that a relation such as $\nu_\alpha e_\alpha=0$ on $\partial\omega$ means that $\nu_1 e_1+\nu_2 e_2=0$ on $\partial\omega$. A relation such as $\partial_\alpha u_3=0$ in $\omega$ means that it holds for all indices $\alpha$ according to the above convention, so it is equivalent to $\partial_1 u_3=\partial_2 u_3=0$ in $\omega$. 

Strong and weak convergences in any normed vector space are respectively denoted 
$\to\text{ and }\rightharpoonup.$

The notation $y=(y_\alpha)$ denotes a generic point in $\capital{R}^2$ and partial derivatives, in the classical or distributional sense, are denoted $\partial_\alpha:= \partial/\partial y_\alpha$ and $\partial_{\alpha\beta}:= \partial^2/\partial y_\alpha \partial y_\beta$.

Vector and matrix fields are denoted by boldface letters. The subspace of $\capital{R}^{2\times 2}$ formed by all symmetric matrices is denoted $\capital{S}^2$. The notation $(c_{\alpha\beta})$ designate the $2\times 2$ matrix whose component at its $\alpha$-row and $\beta$-column is $c_{\alpha\beta}$.

The usual norm of the Lebesgue space $L^p (\omega)$, $1 \leq p \leq \infty$, is denoted by $\| \cdot \|_{L^p(\omega)}$. The notation $L^p (\omega; \capital{R}^{m\times n})$ denotes for every positive integers $m$ and $n$ the space of matrix fields $\boldsymbol{A}=(A_{ij}): \omega \to \capital{R}^{m \times n}$ with components $A_{ij}$ in the Lebesgue space $L^p(\omega)$. This space of matrix-values functions is equipped with the norm
\begin{equation*}
	\|\boldsymbol{A}\|_{L^p(\omega)} := \Big( \sum\limits_{i=1}^m\sum\limits_{j=1}^n \|A_{ij}\|^p_{L^p(\omega)}\Big)^{1/p} \textrm{ for all } \boldsymbol{A} \in L^p (\omega; \capital{R}^{m\times n}) \text{ if } p<\infty,
\end{equation*}
and by $\|\boldsymbol{A}\|_{L^p(\omega)} :=\max_{i,j}\|A_{ij}\|_{L^\infty(\omega)}$ if $p=\infty$. Note that we use the same notation for the norm of scalar, vector, and matrix, valued functions.

The usual norm of the Sobolev space $W^{m,p}(\omega)$, $m \in \capital{N}^*$, $1 \leq p \leq \infty$, is denoted by $\| \cdot \|_{W^{m,p}(\omega)}$. The notations $H^1(\omega)$ and $H^2(\omega)$ respectively denote the spaces $W^{1,2}(\omega)$ and $W^{2,2}(\omega)$.

\section{Existence theorem for rigid plates}
\label{s3}

Let $V(\omega)$ denote the set of all admissible displacements of a nonlinearly elastic plate whose total energy is the functional $J$ defined by \eqref{e1}-\eqref{e6}. 

We prove in this section that if $V(\omega)$ satisfies a certain rigidity assumption (see \eqref{e7}), then the energy functional $J$ has a minimizer in the set $V(\omega)$. The point of this result is that it dissociates the proof of existence of minimizers from the proof of the rigidity of the plate. Thus establishing an existence theorem reduces to verifying whether the boundary conditions imposed on the admissible displacements of the plate imply that the  rigidity assumption \eqref{e7} is satisfied or not. 

To begin with, we recall the following property of the functional $J$, as well as its proof for reader's convenience.

\begin{thm}
	\label{t1}
	Let $\omega\subset{\capital{R}}^2$ be a bounded and connected open set whose boundary is Lipschitz-continuous.  Then the  functional $J:H^1(\omega)\times H^1(\omega)\times H^2(\omega)\to{\capital{R}}$ defined by \eqref{e1}-\eqref{e6} is sequentially weakly lower semicontinuous.
\end{thm}

\begin{proof}
	Since $\omega$ is a bounded Lipschitz in ${\capital{R}}^2$, $H^2(\omega)\subset W^{1,4}(\omega)$ and the embedding is compact; cf. Adams \& Fournier \cite{ada}. Then, if a sequence ${\boldletter{u}}^n=(u^n_i)$, $n\in{\capital{N}}$, and an element ${\boldletter{u}}=(u_i)$ in the space $H^1(\omega)\times H^1(\omega)\times H^2(\omega)$ satisfy  
	\begin{equation*}
		u_\alpha^n \rightharpoonup u_\alpha \text{ \ in } H^1(\omega) \text{ when } n\to \infty  
	\end{equation*}
	and  
	\begin{equation*}
		u_3^n \rightharpoonup u_3 \text{ \ in } H^2(\omega) \text{ when } n\to \infty,
	\end{equation*}
	where $\rightharpoonup$ denotes the convergence with respect to the weak topology of the space in which the convergence takes place, then  
	\begin{equation*}
		u_3^n \to u_3 \text{ \ in } W^{1,4}(\omega) \text{ when } n\to \infty,
	\end{equation*}
	where $\to$ denotes the convergence with respect to the strong topology (induced by the norm of the space). 
	
	Consequently
	\begin{equation*}\aligned{}
		E_{\alpha\beta}({\boldletter{u}}^n)& \rightharpoonup E_{\alpha\beta}({\boldletter{u}})\text{ \ in } L^2(\omega) \text{ when } n\to \infty,\\
		F_{\alpha\beta}({\boldletter{u}}^n)& \rightharpoonup F_{\alpha\beta}({\boldletter{u}})\text{ \ in } L^2(\omega) \text{ when } n\to \infty.
		\endaligned\end{equation*}
	Then 
	\begin{equation*}\aligned{}
		\|{\boldletter{E}}({\boldletter{u}})\|_{\lambda,\mu} &\leq \liminf_{n\to\infty} \|{\boldletter{E}}({\boldletter{u}}^n)\|_{\lambda,\mu},\\
		\|{\boldletter{F}}({\boldletter{u}})\|_{\lambda,\mu} &\leq \liminf_{n\to\infty} \|{\boldletter{F}}({\boldletter{u}}^n)\|_{\lambda,\mu},\\
		L({\boldletter{u}})&= \lim_{n\to\infty}L({\boldletter{u}}_n),
		\endaligned\end{equation*}
	so that, in view of the definition \eqref{e1} of the functional $J$, we have
	\begin{equation*}
		J({\boldletter{u}})\leq \liminf_{n\to\infty} J({\boldletter{u}}^n).
	\end{equation*}
\end{proof}

The functional $J:V(\omega)\to{\capital{R}}$ being sequentially weakly lower semicontinuous on any subset $V(\omega)$ of the Hilbert space $H^1(\omega)\times H^1(\omega)\times H^2(\omega)$ equipped with the induced topology, in order to prove that $J$ has a minimizer in the set $V(\omega)$ it suffices to prove that $V(\omega)$ is sequentially weakly closed in $H^1(\omega)\times H^1(\omega)\times H^2(\omega)$ and that $J$ possesses a minimizing sequence that is bounded in the space $H^1(\omega)\times H^1(\omega)\times H^2(\omega)$.

In the absence of any smallness assumptions of the tangential forces like in Ciarlet \& Destuynder \cite{cd}, the existence or nonexistence of a bounded minimizing sequence for $J$ depends on whether the plate is rigid or not. By definition, a plate is called rigid if the set $V(\omega)$ of all its admissible displacements has the following property:
\begin{equation}\label{e7}\aligned{} 
	{\boldletter{u}} \in V(\omega) \text{ and } {\boldletter{E}}({\boldletter{u}})={\boldletter{0}} \text{ in } \omega \text{ implies } {\boldletter{u}}={\boldletter{0}} \text{ in } \omega.
	\endaligned\end{equation}

If a plate is not rigid, then there exists ${\boldletter{w}}=(w_i)\neq{\boldletter{0}}$ in $V(\omega)$ such that ${\boldletter{E}}({\boldletter{w}})={\boldletter{0}}$. Assuming that $V(\omega)=V_H(\omega)\times V_3(\omega)$ where $V_H(\omega)$ is a proper subspace of $H^1(\omega)\times H^1(\omega)$ and $V_3(\omega)$ is a subspace of $H^2(\omega)$ that does not contain non-zero constant functions, then for each real number $t\in{\capital{R}}$, the vector field ${\boldletter{w}}_t:=(t^2 w_1,t^2 w_2, tw_3)$ belongs to $V(\omega)$ and satisfies  
\begin{equation*}
	{\boldletter{E}}({\boldletter{w}}_t)=t^2{\boldletter{E}}({\boldletter{w}})={\boldletter{0}}.
\end{equation*}
Therefore,
\begin{equation*}
	J({\boldletter{w}}_t):=t^2\frac{\varepsilon^3}6 \|{\boldletter{F}}({\boldletter{w}})\|_{\lambda,\mu}^2-t^2\int_{\omega} p^\varepsilon_\alpha w_\alpha dy - t\int_\omega (p^\varepsilon_3 w_3-q^\varepsilon_\alpha\partial_\alpha w_3) dy  
\end{equation*}
for all $t\in{\capital{R}}$. Consequently, there exist functions $p^\varepsilon_\alpha\in L^2(\omega)$ such that  
\begin{equation*}
	\inf_{{\boldletter{u}}\in V(\omega)} J({\boldletter{u}}) =-\infty.
\end{equation*}
This proves that the existence of minimizers for a nonrigid plate cannot be proved for general applied forces (this means without any restriction on the densities $p^\varepsilon_\alpha\in L^2(\omega)$ of the applied forces). In this sense, the existence theorem of Ciarlet \& Destuynder \cite{cd}, whereby the existence of minimizers of $J$ is proved under the assumption that $\sum_\alpha \|p^\varepsilon_\alpha\|_{L^2(\omega)}$ is small enough, is essentially optimal for nonrigid plates. 

If the plate is rigid, the existence of minimizers has been established by Rabier \cite{rab} for general applied forces, but under the additional assumption that the admissible displacements ${\boldletter{u}}=(u_i)\in V(\omega)$ satisfy the boundary conditions $u_i=\partial_\nu u_3=0$ on $\partial\omega$. 

The next theorem improves this result by removing this additional assumption. Thus, if a plate is rigid, the existence of a minimizer for the functional $J$ can be established in the general case, that is, for general applied forces. Sufficient conditions ensuring the rigidity of the plate, hence the existence of a minimizers, are given in Sections \ref{s4} and \ref{s5}. 

Note that assumption \eqref{e8} in the next theorem means that the space $V_3(\omega)$ does not contain any non-zero affine function. 
\begin{thm}
	\label{t2}
	Let $\omega\subset{\capital{R}}^2$ be a bounded and connected open set whose boundary is Lipschitz-continuous. Let $V(\omega)=V_H(\omega)\times V_3(\omega)$, where $V_H(\omega)$ is closed subspace of $H^1(\omega)\times H^1(\omega)$ and 
	$V_3(\omega)$ is closed subspace of $H^2(\omega)$ such that 
	\begin{equation}\label{e8}\aligned{} 
		u_3 \in V_3(\omega) \textup{ and } \partial_{\alpha\beta}u_3=0 \textup{ in } \omega \textup{ imply } u_3=0 \textup{ in } \omega.
		\endaligned\end{equation}
	
	Assume that $V(\omega)$ satisfies the following rigidity property:
	\begin{equation}\label{e9}\aligned{} 
		{\boldletter{u}} \in V(\omega) \textup{ and } {\boldletter{E}}({\boldletter{u}})={\boldletter{0}} \textup{ in } \omega \textup{ imply } {\boldletter{u}}={\boldletter{0}} \textup{ in } \omega.
		\endaligned\end{equation}
	Then the functional $J$ defined by \eqref{e1}-\eqref{e6} has a minimizer in the set $V(\omega)$.
\end{thm}

\begin{proof}
	The proof consists in adapting Rabier's argument to general displacement fields, which may not satisfy the boundary conditions $u_i=\partial_\nu u_3=0$ on $\partial\omega$ used in Rabier \cite{rab}. We divide the proof into six parts, numbered (i) to (vi), for clarity.
	
	\
	
	{\it (i) Three basic observations.} 
	
	The first observation is that the rigidity assumption \eqref{e9} implies that the subspace $V_H(\omega)$ of $H^1(\omega)\times H^1(\omega)$ satisfies the following property: 
	\begin{equation}\label{e10}\aligned{} 
		(u_\alpha) \in V_H(\omega) \textup{ and } \partial_\alpha u_\beta+\partial_\beta u_\alpha =0 \textup{ in } \omega \textup{ imply } u_\alpha=0 \textup{ in } \omega.
		\endaligned\end{equation}
	For, if $(u_\alpha) \in V_H(\omega)$ satisfies $\partial_\alpha u_\beta+\partial_\beta u_\alpha =0$ in $\omega$, then the displacement field ${\boldletter{u}}:=(u_1,u_2,0)$ belongs to $V(\omega)$ and satisfies ${\boldletter{E}}({\boldletter{u}})={\boldletter{0}}$ in $\omega$. Hence the rigidity assumption \eqref{e9} implies that ${\boldletter{u}}={\boldletter{0}}$, or equivalently, that $u_\alpha=0$ in $\omega$. 
	
	Note that property \eqref{e10} of the space $V_H(\omega)$ implies that the semi-norm
	\begin{equation*}
		(u_\alpha) \in H^1(\omega)\times H^1(\omega) \mapsto \sum_{\alpha,\beta} \|\partial_\alpha u_\beta+\partial_\beta u_\alpha\|_{L^2(\omega)}\in{\capital{R}}
	\end{equation*}
	becomes a norm equivalent to the norm of the space $H^1(\omega)\times H^1(\omega)$ when it is restricted to the subspace $V_H(\omega)$ of $H^1(\omega)\times H^1(\omega)$, thanks to the open mapping theorem (see, e.g., Ciarlet \cite{c2025}).
	
	So there exists a constant $C_1\in{\capital{R}}$ such that 
	\begin{equation}\label{e11}\aligned{} 
		\sum_\alpha \|u_\alpha\|_{H^1(\omega)}^2
		& \leq C_1 \sum_{\alpha,\beta} \|\partial_\alpha u_\beta+\partial_\beta u_\alpha\|_{L^2(\omega)}^2 && \text{for all }(u_\alpha)\in V_H(\omega).
		\endaligned\end{equation}
	
	The second observation is that assumption \eqref{e8} of the theorem is equivalent to the assumption that the semi-norm
	\begin{equation*}
		u_3 \in H^2(\omega) \mapsto \sum_{\alpha,\beta} \|\partial_{\alpha\beta} u_3\|_{L^2(\omega)}\in{\capital{R}}
	\end{equation*}
	becomes a norm equivalent to the norm of the space $H^2(\omega)$ when restricted to the subspace $V_3(\omega)$ of $H^2(\omega)$. This equivalence can be proved by using the open mapping theorem (see, e.g., Ciarlet \cite{c2025}). So there exists a constant $C_2\in{\capital{R}}$ such that 
	\begin{equation}\label{e12}\aligned{} 
		\|u_3\|_{H^2(\omega)}^2
		& \leq C_2 \sum_{\alpha,\beta} \|\partial_{\alpha\beta} u_3\|_{L^2(\omega)}^2 && \text{for all } u_3\in V_3(\omega).
		\endaligned\end{equation}
	
	The third observation is that for each minimizing sequence  ${\boldletter{u}}^n$ of the functional $J: V(\omega)\to{\capital{R}}$, there exists an integer $n_0\in{\capital{N}}$ such that, for all $n\geq n_0$,  
	\begin{equation*}
		J({\boldletter{u}}^n)\leq J({\boldletter{0}})=0.
	\end{equation*}
	Therefore,  
	\begin{equation*}
		\frac{\varepsilon}2 \|{\boldletter{E}}({\boldletter{u}}^n)\|_{\lambda,\mu}^2+\frac{\varepsilon^3}6 \|{\boldletter{F}}({\boldletter{u}}^n)\|_{\lambda,\mu}^2
		\leq L({\boldletter{u}}^n)
	\end{equation*}
	for all $n\geq n_0$. Combined with the definition of the norm $\|\cdot\|_{\lambda,\mu}$, this implies that
	\begin{equation}\label{e13}\aligned{} 
		\frac{\mu\varepsilon}8 \|(\partial_\alpha u^n_\beta+\partial_\beta u^n_\alpha+\partial_\alpha u^n_3\partial_\beta u^n_3)\|_{L^2(\omega)}^2+\frac{2\mu\varepsilon^3}3 \|(\partial_{\alpha\beta}u^n_3)\|_{L^2(\omega)}^2\leq L({\boldletter{u}}^n),
		\endaligned\end{equation}
	on the one hand. 
	
	The definition of the linear form $L$ implies that 
	\begin{equation*}\aligned{}
		L({\boldletter{u}}^n)&=\int_{\omega} (p^\varepsilon_i u^n_i- q^\varepsilon_\alpha\partial_\alpha u^n_3) dy
		\leq \sum_\alpha \|p^\varepsilon_\alpha\|_{L^2(\omega)}\|u^n_\alpha\|_{L^2(\omega)}
		\\& 
		+ \Big(\|p^\varepsilon_3\|_{L^2(\omega)}^2+ \sum_\alpha \|q^\varepsilon_\alpha\|_{L^2(\omega)}^2\Big)^{1/2}\|u^n_3\|_{H^1(\omega)},
		\endaligned\end{equation*}
	which combined with inequalities \eqref{e11} and \eqref{e12} next implies that 
	\begin{equation*}\aligned{}
		L({\boldletter{u}}^n)
		& \leq \sqrt{C_1}\|(p^\varepsilon_\alpha)\|_{L^2(\omega)} \|(\partial_\alpha u^n_\beta+\partial_\beta u^n_\alpha)\|_{L^2(\omega)}
		\\
		&+ \sqrt{C_2}\Big(\|p^\varepsilon_3\|_{L^2(\omega)}^2+ \|(q^\varepsilon_\alpha)\|_{L^2(\omega)}^2\Big)^{1/2}\|(\partial_{\alpha\beta}u^n_3)\|_{L^2(\omega)},\\
		& \leq \sqrt{C_1}\|(p^\varepsilon_\alpha)\|_{L^2(\omega)}  \|(\partial_\alpha u^n_\beta+\partial_\beta u^n_\alpha+\partial_\alpha u^n_3\partial_\beta u^n_3)\|_{L^2(\omega)}\\
		&
		+\sqrt{C_1}\|(p^\varepsilon_\alpha)\|_{L^2(\omega)} \|(\partial_\alpha u^n_3\partial_\beta u^n_3)\|_{L^2(\omega)}\\
		&
		+ \sqrt{C_2}\Big(\|p^\varepsilon_3\|_{L^2(\omega)}^2+ \|(q^\varepsilon_\alpha)\|_{L^2(\omega)}^2\Big)^{1/2}\|(\partial_{\alpha\beta}u^n_3)\|_{L^2(\omega)}.
		\endaligned\end{equation*}
	Then
	\begin{equation}\label{e14}\aligned{} 
		L({\boldletter{u}}^n)&\leq 
		\frac{\mu\varepsilon}{16}  \|(\partial_\alpha u^n_\beta+\partial_\beta u^n_\alpha+\partial_\alpha u^n_3\partial_\beta u^n_3)\|_{L^2(\omega)}^2
		+\frac{4C_1}{\mu\varepsilon}\|(p^\varepsilon_\alpha)\|_{L^2(\omega)}^2 
		\\
		&+\sqrt{2C_1}\|(p^\varepsilon_\alpha)\|_{L^2(\omega)} \|(\partial_\alpha u^n_3)\|_{L^4(\omega)}^2\\
		&+\frac{\mu\varepsilon^3}{3}\|(\partial_{\alpha\beta}u^n_3)\|_{L^2(\omega)}^2
		+  \frac{3C_2}{4\mu\varepsilon^3}\Big(\|p^\varepsilon_3\|_{L^2(\omega)}^2+ \|(q^\varepsilon_\alpha)\|_{L^2(\omega)}^2\Big),
		\endaligned\end{equation}
	on the other hand.
	
	Combining inequality \eqref{e14} with inequality \eqref{e13} yields 
	\begin{equation}\label{e15}\aligned{} 
		\frac{\mu\varepsilon}{16} \|(\partial_\alpha u^n_\beta+\partial_\beta u^n_\alpha+\partial_\alpha u^n_3\partial_\beta u^n_3)\|_{L^2(\omega)}^2+\frac{\mu\varepsilon^3}3 \|(\partial_{\alpha\beta}u^n_3)\|_{L^2(\omega)}^2\\
		\leq
		A+B\|(\partial_\alpha u^n_3)\|_{L^4(\omega)}^2,
		\endaligned\end{equation}
	where 
	\begin{equation}\label{e16}\aligned{} 
		A&:=\frac{4C_1}{\mu\varepsilon}\|(p^\varepsilon_\alpha)\|_{L^2(\omega)}^2+ \frac{3C_2}{4\mu\varepsilon^3}\Big(\|p^\varepsilon_3\|_{L^2(\omega)}^2+ \|(q^\varepsilon_\alpha)\|_{L^2(\omega)}^2\Big),\\
		B&:=\sqrt{2C_1}\|(p^\varepsilon_\alpha)\|_{L^2(\omega)}.
		\endaligned\end{equation}
	
	\
	
	{\it (ii) We prove here that if all minimizing sequences of  $J:V(\omega)\to{\capital{R}}$ are unbounded in the space $H^1(\omega)\times H^1(\omega)\times H^2(\omega)$, then all minimizing sequences ${\boldletter{u}}^n=(u^n_i)\in V(\omega)$, $n\in{\capital{N}}$, of $J$ must satisfy  
		\begin{equation*}
			\|(u^n_\alpha)\|_{H^1(\omega)} \to \infty \text{ \ when } n\to\infty,
		\end{equation*}
		and
		\begin{equation*}
			\|u^n_3\|_{H^2(\omega)} \to \infty \text{ \ when } n\to\infty.
		\end{equation*}
	}
	
	To prove this, assume first that there exists a minimizing sequence ${\boldletter{u}}^n$ of $J$ in $V(\omega)$ such that 
	\begin{equation}\label{e17}\aligned{} 
		\sup_{n\in{\capital{N}}} \|(u^n_\alpha)\|_{H^1(\omega)} <\infty.
		\endaligned\end{equation}
	Since ${\boldletter{u}}^n$ is a minimizing sequence, there exists $n_0\in{\capital{N}}$ such that, for all $n\geq n_0$,  
	\begin{equation*}
		J({\boldletter{u}}^n)\leq J({\boldletter{0}})=0.
	\end{equation*}
	Therefore, for all $n\geq n_0$, 
	\begin{equation*}\aligned{}
		& \frac{\varepsilon}2 \|{\boldletter{E}}({\boldletter{u}}^n)\|_{\lambda,\mu}^2+\frac{\varepsilon^3}6 \|{\boldletter{F}}({\boldletter{u}}^n)\|_{\lambda,\mu}^2
		\leq L({\boldletter{u}}^n)=\int_{\omega} (p^\varepsilon_i u^n_i- q^\varepsilon_\alpha\partial_\alpha u^n_3) dy\\
		& \leq \sum_\alpha \|p^\varepsilon_\alpha\|_{L^2(\omega)}\|u^n_\alpha\|_{L^2(\omega)}
		+ \Big(\|p^\varepsilon_3\|_{L^2(\omega)}^2+ \sum_\alpha \|q^\varepsilon_\alpha\|_{L^2(\omega)}^2\Big)^{1/2}\|u^n_3\|_{H^1(\omega)}. 
		\endaligned\end{equation*}
	Using inequality \eqref{e17} in the right-hand side above implies that there exist two constants $A_0$ and $B_0$ independent of $n$ such that, for all $n\geq n_0$,
	\begin{equation*}
		\frac{\varepsilon^3}6 \|{\boldletter{F}}({\boldletter{u}}^n)\|_{\lambda,\mu}^2
		\leq A_0+B_0 \|u^n_3\|_{H^1(\omega)}.
	\end{equation*}
	On the other hand, the definition of the norm $\|\cdot\|_{\lambda,\mu}$ and inequality \eqref{e12} imply that  
	\begin{equation*}
		\|{\boldletter{F}}({\boldletter{u}}^n)\|_{\lambda,\mu}^2 \geq 4\mu \sum_{\alpha,\beta}\|\partial_{\alpha\beta} u^n_3\|_{L^2(\omega)}^2
		\geq \frac{4\mu}{C_2} \|u^n_3\|_{H^2(\omega)}^2,
	\end{equation*}
	where $C_2>0$ is a constant. Combined with the previous inequality, this implies that  
	\begin{equation*}
		\sup_{n\in{\capital{N}}} \|u^n_3\|_{H^2(\omega)} <\infty.
	\end{equation*}
	
	But this inequality, combined with inequality \eqref{e17}, implies that the minimizing sequence ${\boldletter{u}}^n$ is bounded in the space $H^1(\omega)\times H^1(\omega)\times H^2(\omega)$, which contradicts the assumption that such a minimizing sequence does not exist. 
	
	Second, assume that there exists a minimizing sequence ${\boldletter{u}}^n$ of $J$ in $V(\omega)$ such that 
	\begin{equation}\label{e18}\aligned{} 
		\sup_{n\in{\capital{N}}} \|u^n_3\|_{H^2(\omega)} <\infty.
		\endaligned\end{equation}
	Using that $J({\boldletter{u}}^n)\leq 0$, we deduce that (in the same way as above) 
	\begin{equation*}\aligned{}
		& \frac{\varepsilon}2 \|{\boldletter{E}}({\boldletter{u}}^n)\|_{\lambda,\mu}^2+\frac{\varepsilon^3}6 \|{\boldletter{F}}({\boldletter{u}}^n)\|_{\lambda,\mu}^2\\
		& \leq \sum_\alpha \|p^\varepsilon_\alpha\|_{L^2(\omega)}\|u^n_\alpha\|_{L^2(\omega)}
		+ \Big(\|p^\varepsilon_3\|_{L^2(\omega)}^2+ \sum_\alpha \|q^\varepsilon_\alpha\|_{L^2(\omega)}^2\Big)^{1/2}\|u^n_3\|_{H^1(\omega)}, 
		\endaligned\end{equation*}
	then that 
	\begin{equation}\label{e19}\aligned{} 
		\frac{\varepsilon}2 \|{\boldletter{E}}({\boldletter{u}}^n)\|_{\lambda,\mu}^2
		\leq A_1 \sum_\alpha\|u^n_\alpha\|_{L^2(\omega)}+ B_1,
		\endaligned\end{equation}
	for some constants $A_1,B_1\in{\capital{R}}$. 
	
	On the other hand, the definition of the norm $\|\cdot\|_{\lambda,\mu}$ implies  that 
	\begin{equation*}\aligned{}
		\|{\boldletter{E}}({\boldletter{u}}^n)\|_{\lambda,\mu}^2
		&\geq 4 \mu\sum_{\alpha,\beta}\|E_{\alpha\beta}({\boldletter{u}}^n)\|_{L^2(\omega)}^2\\
		& = \mu \sum_{\alpha,\beta}\|\partial_\alpha u^n_\beta+\partial_\beta u^n_\alpha+\partial_\alpha u^n_3\partial_\beta u^n_3\|_{L^2(\omega)}^2\\
		& \geq \frac{\mu}{2} \sum_{\alpha,\beta}\|\partial_\alpha u^n_\beta+\partial_\beta u^n_\alpha\|_{L^2(\omega)}^2
		- \mu\sum_{\alpha,\beta}\|\partial_\alpha u^n_3\partial_\beta u^n_3\|_{L^2(\omega)}^2,
		\endaligned\end{equation*}
	which combined with inequality \eqref{e11} shows that  
	\begin{equation*}
		\|{\boldletter{E}}({\boldletter{u}}^n)\|_{\lambda,\mu}^2 \geq \frac{\mu}{2C_1}\sum_{\alpha} \|u^n_\alpha\|_{H^1(\omega)}^2-\mu\sum_{\alpha,\beta}\|\partial_\alpha u^n_3\partial_\beta u^n_3\|_{L^2(\omega)}^2,
	\end{equation*}
	where $C_1>0$ is a constant. Then we infer from inequality \eqref{e19} that  
	\begin{equation*}
		\frac{\mu\varepsilon}{4C_1}\sum_{\alpha} \|u^n_\alpha\|_{H^1(\omega)}^2
		\leq A_1 \sum_\alpha\|u^n_\alpha\|_{L^2(\omega)}+\Big(B_1+\frac{\mu\varepsilon}{2} \sum_{\alpha,\beta}\|\partial_\alpha u^n_3\partial_\beta u^n_3\|_{L^2(\omega)}^2\Big),  
	\end{equation*}
	which combined with assumption \eqref{e18} next implies that
	\begin{equation*}
		\sup_{n\in{\capital{N}}} \left( \sum_\alpha \|u^n_\alpha\|_{H^1(\omega)} \right)<\infty.
	\end{equation*}
	But this inequality, combined with inequality \eqref{e18}, implies that the minimizing sequence ${\boldletter{u}}^n$ is bounded in the space $H^1(\omega)\times H^1(\omega)\times H^2(\omega)$, which contradicts the assumption that such a minimizing sequence does not exist. 
	
	\
	
	{\it (iii) We prove here that if all minimizing sequences of  $J:V(\omega)\to{\capital{R}}$ are unbounded in the space $H^1(\omega)\times H^1(\omega)\times H^2(\omega)$, then every minimizing sequence ${\boldletter{u}}^n=(u^n_i)\in V(\omega)$, $n\in{\capital{N}}$, of $J$ satisfies the following property: There exist constants $n_0\in{\capital{N}}$ and $B_3\in{\capital{R}}$ such that 
		\begin{equation}\label{e20}\aligned{} 
			\|(\partial_{\alpha\beta}u^n_3)\|_{L^2(\omega)}
			& \leq
			2 B_3\|(\partial_\alpha u^n_3)\|_{L^4(\omega)} \text{ \ for all } n\geq n_0.
			\endaligned
		\end{equation}
	}
	
	We proved in part (ii) of the proof that if the functional $J$ does not possess any minimizing sequence that is bounded in the space $H^1(\omega)\times H^1(\omega)\times H^2(\omega)$, then all minimizing sequences ${\boldletter{u}}^n=(u^n_i)\in V(\omega)$, $n\in{\capital{N}}$, of $J$ must satisfy 
	\begin{equation}\label{e21}\aligned{} 
		&\|(u^n_\alpha)\|_{H^1(\omega)} \to \infty && \text{ \ when } n\to\infty,
		\\
		& \|u^n_3\|_{H^2(\omega)} \to \infty && \text{ \ when } n\to\infty.
		\endaligned\end{equation}
	Besides, as shown by inequality \eqref{e15} established above in part (i) of the proof, they must also satisfy, for all $n\in{\capital{N}}$, 
	\begin{equation}\label{e22}\aligned{} 
		\|(\partial_\alpha u^n_\beta+\partial_\beta u^n_\alpha+\partial_\alpha u^n_3\partial_\beta u^n_3)\|_{L^2(\omega)}
		& \leq
		A_2+B_2\|(\partial_\alpha u^n_3)\|_{L^4(\omega)},\\
		\|(\partial_{\alpha\beta}u^n_3)\|_{L^2(\omega)}
		& \leq
		A_3+B_3\|(\partial_\alpha u^n_3)\|_{L^4(\omega)},
		\endaligned\end{equation}
	for some constants $A_2,B_2,A_3,B_3$ defined in terms of $\mu$, $\varepsilon$, and of the constants $A$ and $B$ given by \eqref{e16}.
	
	Using inequality \eqref{e12} and the last convergence in \eqref{e21}, we deduce from the last inequality of \eqref{e22} that  
	\begin{equation*}
		A_3+B_3\|(\partial_\alpha u^n_3)\|_{L^4(\omega)}\geq \frac{1}{\sqrt{C_2}}\|u^n_3\|_{H^2(\omega)}\to \infty \text{ \ when } n\to\infty.
	\end{equation*}
	Hence 
	\begin{equation}\label{e23}\aligned{} 
		\|(\partial_\alpha u^n_3)\|_{L^4(\omega)}\to \infty \text{ \ when } n\to\infty,
		\endaligned\end{equation}
	which combined with the second inequality of \eqref{e22} shows that there exists an integer $n_0\in{\capital{N}}$ such that  
	\begin{equation*}
		\|(\partial_{\alpha\beta}u^n_3)\|_{L^2(\omega)}\leq
		2 B_3\|(\partial_\alpha u^n_3)\|_{L^4(\omega)} \text{ \ for all } n\geq n_0.
	\end{equation*}
	This is precisely inequality \eqref{e20}, which is thus proved.

	\
	
	{\it (iv) We prove here that if all minimizing sequences of  $J:V(\omega)\to{\capital{R}}$ are unbounded in the space $H^1(\omega)\times H^1(\omega)\times H^2(\omega)$,  then every minimizing sequence ${\boldletter{u}}^n=(u^n_i)\in V(\omega)$, $n\in{\capital{N}}$, of $J$ satisfies the following property: There exist a constant $C_3\in {\capital{R}}$ such that 
		\begin{equation}\label{e24}\aligned{} 
			\|(u^n_\alpha)\|_{H^1(\omega)}\leq C_3\|(\partial_\alpha u^n_\beta+\partial_\beta u^n_\alpha+\partial_\alpha u^n_3\partial_\beta u^n_3)\|_{L^2(\omega)}  \text{ \ for all } n\in{\capital{N}}.
			\endaligned\end{equation}
	}
	
	Assume for the sake of contradiction that this is false. Then there exists a minimizing sequence ${\boldletter{u}}^n=(u^n_i)\in V(\omega)$, $n\in{\capital{N}}$, of $J$ and a strictly increasing function $\sigma:{\capital{N}}\to{\capital{N}}$ such that  
	\begin{equation*}
		\|(u^{\sigma(k)}_\alpha)\|_{H^1(\omega)}> k \|(\partial_\alpha u^{\sigma(k)}_\beta+\partial_\beta u^{\sigma(k)}_\alpha+\partial_\alpha u^{\sigma(k)}_3\partial_\beta u^{\sigma(k)}_3)\|_{L^2(\omega)}
	\end{equation*}
	for all $k\in{\capital{N}}$. Define  
	\begin{equation*}
		w^{\sigma(k)}_\alpha:=\frac{u^{\sigma(k)}_\alpha}{\|(u^{\sigma(k)}_\alpha)\|_{H^1(\omega)}}
		\text{ \ and \ } 
		w^{\sigma(k)}_3:=\frac{u^{\sigma(k)}_3}{\|(u^{\sigma(k)}_\alpha)\|_{H^1(\omega)}^{1/2}}
	\end{equation*}
	for all $k\in {\capital{N}}$. Then
	\begin{equation}\label{e25}\aligned{} 
		\left\|(\partial_\alpha w^{\sigma(k)}_\beta+\partial_\beta w^{\sigma(k)}_\alpha+\partial_\alpha w^{\sigma(k)}_3\partial_\beta w^{\sigma(k)}_3)\right\|_{L^2(\omega)} & < \frac1k &&\text{ for all }k\in{\capital{N}},\\
		\|(w^{\sigma(k)}_\alpha)\|_{H^1(\omega)}& =1 &&\text{ for all }k\in{\capital{N}},
		\endaligned\end{equation}
	and, in view of inequality \eqref{e20},
	\begin{equation}\label{e26}\aligned{} 
		\|(\partial_{\alpha\beta}w^{\sigma(k)}_3)\|_{L^2(\omega)}
		& \leq
		2 B_3\left\|(\partial_\alpha w^{\sigma(k)}_3)\right\|_{L^4(\omega)} \text{ \ for all } k\geq n_0.
		\endaligned\end{equation}
	
	Next, relations \eqref{e25} and the triangle inequality imply that  
	\begin{equation*}
		\left\|(\partial_\alpha w^{\sigma(k)}_3\partial_\beta w^{\sigma(k)}_3)\right\|_{L^2(\omega)}
		\leq \frac1k+2\leq 3 \text{ \ for all }k\geq 1,
	\end{equation*}
	which in turn implies that  
	\begin{equation*}
		\left\|(\partial_\alpha w^{\sigma(k)}_3)\right\|_{L^4(\omega)}
		\leq \sqrt3 \text{ \ for all }k\geq 1,
	\end{equation*}
	Then we infer from inequality \eqref{e26} that 
	\begin{equation}\label{e27}\aligned{} 
		\|(\partial_{\alpha\beta}w^{\sigma(k)}_3)\|_{L^2(\omega)}
		& \leq
		2 \sqrt3 B_3 \text{ \ for all } k\geq n_0,
		\endaligned\end{equation}
	which proves that the sequence $(w^{\sigma(k)}_3)_{k\in{\capital{N}}}$ is bounded in $H^2(\omega)$ thanks to inequality \eqref{e12}. 
	
	Since the sequence $(w^{\sigma(k)}_\alpha)_{k\in{\capital{N}}}$ is also bounded in $H^1(\omega)\times H^1(\omega)$ thanks to the second inequality of  \eqref{e25}, and since the subspaces $V_H(\omega)$ and $V_3(\omega)$ are closed respectively in $H^1(\omega)\times H^1(\omega)$ and $H^2(\omega)$ by the assumptions of the theorem, there exists a strictly increasing function $\psi: {\capital{N}}\to {\capital{N}}$ and two elements $(w_\alpha)\in V_H(\omega)$ and $w_3\in V_3(\omega)$ such that 
	\begin{equation*}\aligned{}
		w^{\sigma(\psi(p))}_\alpha & \rightharpoonup w_\alpha && \text{ in } H^1(\omega) \text{ \ when } p\to\infty,\\
		w^{\sigma(\psi(p))}_3 & \rightharpoonup w_3 && \text{ in } H^2(\omega) \text{ \ when } p\to\infty,
		\endaligned\end{equation*}
	where $\rightharpoonup$ denotes the weak convergence. Since the embeddings $H^1(\omega)\subset  L^4(\omega)$ and $H^1(\omega)\subset  L^2(\omega)$ are compact thanks to the assumptions on the set $\omega$ (see, e.g., Adams \& Fournier \cite{ada}), we also have 
	\begin{equation}\label{e28}\aligned{} 
		w^{\sigma(\psi(p))}_\alpha & \to w_\alpha && \text{ in } L^2(\omega) \text{ \ when } p\to\infty,\\
		\partial_\alpha w^{\sigma(\psi(p))}_3 & \to \partial_\alpha w_3 && \text{ in } L^4(\omega) \text{ \ when } p\to\infty,
		\endaligned\end{equation}
	where $\to$ denotes the strong convergence. 
	
	Consequently,  
	\begin{equation*}
		\partial_\alpha w^{\sigma(\psi(p))}_\beta+\partial_\beta w^{\sigma(\psi(p))}_\alpha+\partial_\alpha w^{\sigma(\psi(p))}_3\partial_\beta w^{\sigma(\psi(p))}_3\rightharpoonup \partial_\alpha w_\beta+\partial_\beta w_\alpha+\partial_\alpha w_3\partial_\beta w_3
	\end{equation*}
	in $L^2(\omega)$ when $p\to\infty$. Then we infer from inequality \eqref{e25} that  
	\begin{equation*}
		\|\partial_\alpha w_\beta+\partial_\beta w_\alpha+\partial_\alpha w_3\partial_\beta w_3\|_{L^2(\omega)}
		\leq \liminf_{p\to\infty}\frac1{\psi(p)}=0,
	\end{equation*}
	which implies that  
	\begin{equation*}
		w_\alpha=w_3=0,  
	\end{equation*}
	thanks to the rigidity assumption \eqref{e9} of the theorem. Using now the strong convergence $\partial_\alpha w^{\sigma(\psi(p))}_3 \to \partial_\alpha w_3=0$ in $L^4(\omega)$ established in \eqref{e28}, we deduce from inequality \eqref{e25} that 
	\begin{equation}\label{e29}\aligned{} 
		\left\|(\partial_\alpha w^{\sigma(\psi(p))}_\beta+\partial_\beta w^{\sigma(\psi(p))}_\alpha)\right\|_{L^2(\omega)} 
		& \leq 
		\frac1{\psi(p)}+
		\left\|(\partial_\alpha w^{\sigma(\psi(p))}_3\partial_\beta w^{\sigma(\psi(p))}_3)\right\|_{L^2(\omega)} \\
		& \leq 
		\frac1{\psi(p)}+
		2 \left\|(\partial_\alpha w^{\sigma(\psi(p))}_3)\right\|_{L^4(\omega)}^2\\
		&\to 0 \text{ \ when  } p\to\infty. 
		\endaligned\end{equation}
	Then inequality \eqref{e11} implies that  
	\begin{equation*}
		w^{\sigma(\psi(p)}\to 0 \text{ \ in } H^1(\omega) \text{ \ when  } p\to\infty,  
	\end{equation*}
	or equivalently, that  
	\begin{equation*}
		\|(w^{\sigma(\psi(p))}_\alpha)\|_{H^1(\omega)}\to 0  \text{ \ when  } p\to\infty.
	\end{equation*}
	This contradicts the second relation of \eqref{e25}, according to which \begin{equation*}\aligned{}
		\|(w^{\sigma(\psi(p))}_\alpha)\|_{H^1(\omega)}& =1 &&\text{ for all }p\in{\capital{N}}.
		\endaligned\end{equation*}
	This completes the proof by contradiction of inequality \eqref{e24}.
	
	\
	
	{\it (v) We prove here that the functional $J:V(\omega)\to{\capital{R}}$ possesses a minimizing sequence that is bounded in the space $H^1(\omega)\times H^1(\omega)\times H^2(\omega)$.}
	
	Assume for the sake of contradiction that this is false. Then all minimizing sequences of $J$ are unbounded in the space $H^1(\omega)\times H^1(\omega)\times H^2(\omega)$. Let ${\boldletter{u}}^n=(u^n_i)\in V(\omega)$, $n\in{\capital{N}}$, be a minimizing sequence of $J$. We have established in part (ii) of the proof that 
	\begin{equation}\label{e30}\aligned{} 
		\left\|(u^n_\alpha)\right\|_{H^1(\omega)} & \to \infty \text{ \ when } n\to\infty,\\
		\left\|u^n_3\right\|_{H^2(\omega)} & \to \infty \text{ \ when } n\to\infty,
		\endaligned\end{equation}
	we have proved in part (iii) that (see \eqref{e20})
	\begin{equation*}\aligned{} 
		\|(\partial_{\alpha\beta}u^n_3)\|_{L^2(\omega)}
		& \leq
		2 B_3\|(\partial_\alpha u^n_3)\|_{L^4(\omega)} \text{ \ for all } n\geq n_0
		\endaligned\end{equation*}
	for some constants $n_0\in{\capital{N}}$ and $B_3\in {\capital{R}}$, and we have proved in parts (i) and (iv) that,  for all $n\in{\capital{N}}$ (see \eqref{e15} and \eqref{e24}), 
	\begin{equation}\label{e31}\aligned{} 
		\frac{\mu\varepsilon}{16} \|(\partial_\alpha u^n_\beta+\partial_\beta u^n_\alpha+\partial_\alpha u^n_3\partial_\beta u^n_3)\|_{L^2(\omega)}^2
		\leq
		A+B\|(\partial_\alpha u^n_3)\|_{L^4(\omega)}^2,
		\endaligned\end{equation}
	and
	\begin{equation*}\aligned{}
		\|(u^n_\alpha)\|_{H^1(\omega)}\leq C_3\|(\partial_\alpha u^n_\beta+\partial_\beta u^n_\alpha+\partial_\alpha u^n_3\partial_\beta u^n_3)\|_{L^2(\omega)}
		\endaligned\end{equation*}
	for some constants $A,B,C_3\in {\capital{R}}$ . Note that the last two inequalities together imply that  
	\begin{equation*}
		\|(u^n_\alpha)\|_{H^1(\omega)}\leq A_4+B_4\|(\partial_\alpha u^n_3)\|_{L^4(\omega)}  \text{ \ for all } n\in{\capital{N}},
	\end{equation*}
	for some constants $A_4,B_4\in {\capital{R}}$.
	
	In view of \eqref{e30}, we may assume that $\|u^n_3\|_{H^2(\omega)}\geq 1$ for all $n\in{\capital{N}}$. Then, for each $n\in{\capital{N}}$, we define  
	\begin{equation*}
		z^n_\alpha:=\frac{u^n_\alpha}{\|u^n_3\|_{H^2(\omega)}} 
		\text{ \ and } 
		z^n_3:=\frac{u^n_3}{\|u^n_3\|_{H^2(\omega)}}.  
	\end{equation*}
	Then 
	\begin{equation}\label{e32}\aligned{} 
		\|z^n_3\|_{H^2(\omega)}=1 \text{ \ for all } n\in{\capital{N}},
		\endaligned\end{equation}
	and
	\begin{equation}\label{e33}\aligned{} 
		\|(\partial_{\alpha\beta}z^n_3)\|_{L^2(\omega)}
		& \leq
		2 B_3\|(\partial_\alpha z^n_3)\|_{L^4(\omega)} \text{ \ for all } n\geq n_0.
		\endaligned\end{equation}
	Moreover, 
	\begin{equation*}\aligned{}
		\|(z^n_\alpha)\|_{H^1(\omega)}
		& \leq \frac{A_4}{\|u^n_3\|_{H^2(\omega)}} +B_4\|(\partial_\alpha z^n_3)\|_{L^4(\omega)}\\
		& \leq A_4 +B_4\|(\partial_\alpha z^n_3)\|_{L^4(\omega)}
		\text{ \ for all } n\in{\capital{N}},
		\endaligned\end{equation*}
	which, using that the embedding $H^1(\omega)\subset L^4(\omega)$ is continuous and that the sequence $(z^n_3)$ is bounded in $H^2(\omega)$ (see \eqref{e32}), next implies that 
	\begin{equation}\label{e34}\aligned{} 
		\|(z^n_\alpha)\|_{H^1(\omega)}
		& \leq A_4 +B_5\|z^n_3\|_{H^2(\omega)}=A_4+B_5
		\text{ \ for all } n\in{\capital{N}}.
		\endaligned\end{equation}
	Thus we have proved that the sequence $(z^n_3)_{n\in{\capital{N}}}$ is bounded in $H^2(\omega)$ and that the sequences $(z^n_\alpha)_{n\in{\capital{N}}}$ are bounded in $H^1(\omega)$. 
	
	Furthermore, inequality \eqref{e31} implies that 
	\begin{equation*}\aligned{}
		\|(\partial_\alpha u^n_3\partial_\beta u^n_3)\|_{L^2(\omega)}
		& \leq
		\sqrt{\frac{16A}{\mu\varepsilon}}+ \sqrt{\frac{16B}{\mu\varepsilon}}\|(\partial_\alpha u^n_3)\|_{L^4(\omega)}\\
		& + \|(\partial_\alpha u^n_\beta+\partial_\beta u^n_\alpha)\|_{L^2(\omega)},
		\endaligned\end{equation*}
	which in turn implies that 
	\begin{equation*}\aligned{}
		\|u^n_3\|_{H^2(\omega)}\|(\partial_\alpha z^n_3)\|_{L^4(\omega)}^2
		& \leq
		\frac{A_5}{\|u^n_3\|_{H^2(\omega)}}+ B_6\|(\partial_\alpha z^n_3)\|_{L^4(\omega)}\\
		&+ \|(\partial_\alpha z^n_\beta+\partial_\beta z^n_\alpha)\|_{L^2(\omega)},
		\endaligned\end{equation*}
	for some constants $A_5\in{\capital{R}}$ and $B_6\in{\capital{R}}$.  Since $\|u^n_3\|_{H^2(\omega)}\to \infty$ when $n\to\infty$ and the sequences $(\partial_\alpha z^n_3)_{n\in{\capital{N}}}$ and $(\partial_\alpha z^n_\beta)_{n\in{\capital{N}}}$ are bounded respectively in $L^4(\omega)$ and $L^2(\omega)$, the above inequality implies that  
	\begin{equation*}
		\|(\partial_\alpha z^n_3)\|_{L^4(\omega)}^2 \to 0 \text{ \ when } n\to\infty.
	\end{equation*}
	Then inequality \eqref{e33} implies that 
	\begin{equation*}\aligned{}
		\|(\partial_{\alpha\beta}z^n_3)\|_{L^2(\omega)}
		\to 0 \text{ \ when } n\to\infty,
		\endaligned\end{equation*}
	which combined with inequality \eqref{e12} implies that 
	\begin{equation}\label{e35}\aligned{} 
		\|z^n_3\|_{H^2(\omega)}
		\to 0 \text{ \ when } n\to\infty.
		\endaligned\end{equation}
	
	But this contradicts the relation $\|z^n_3\|_{H^2(\omega)}=1$ for all $n\in{\capital{N}}$ (see \eqref{e32}). 
	Thus, the functional $J:V(\omega)\to{\capital{R}}$ does possess a minimizing sequence that is bounded in the space $H^1(\omega)\times H^1(\omega)\times H^2(\omega)$. 
	
	\
	
	{\it (vi) We prove here that the functional $J$ defined by \eqref{e1} has a minimizer in the set $V(\omega)$.} 
	
	Let ${\boldletter{u}}^n=(u^n_i)\in V(\omega)=V_H(\omega)\times V_3(\omega)$, $n\in{\capital{N}}$, be a minimizing sequence of $J$ that is bounded in the space $H^1(\omega)\times H^1(\omega)\times H^2(\omega)$ (the existence of such a sequence has been proved in part (v) above).
	
	Since $H^1(\omega)\times H^1(\omega)$ and $H^2(\omega)$ are Hilbert spaces and $V_H(\omega)\subset H^1(\omega)\times H^1(\omega)$ and $ V_3(\omega)\subset H^2(\omega)$ are closed subspaces, there exists a strictly increasing function $\sigma:{\capital{N}}\to{\capital{N}}$ and an element ${\boldletter{u}}=(u_i)\in V(\omega)$ such that 
	\begin{equation*}\aligned{} 
		u^{\sigma(k)}_\alpha&\rightharpoonup u_\alpha \text{ in } H^1(\omega) \text{ when } k\to\infty,\\
		u^{\sigma(k)}_3&\rightharpoonup u_3 \text{ in } H^2(\omega) \text{ when } k\to\infty.
		\endaligned\end{equation*}
	The functional $J:H^1(\omega)\times H^1(\omega)\times H^2(\omega)\to{\capital{R}}$ being sequentially weakly lower semicontinuous (see Theorem~\ref{t1}), we have  
	\begin{equation*}
		J({\boldletter{u}})\leq \liminf_{n\to\infty} J({\boldletter{u}}^n)=\inf_{{\boldletter{w}}\in V(\omega)} J({\boldletter{w}}).
	\end{equation*}
	This completes the proof of the theorem.
\end{proof}

Theorem \ref{t2} proves that the two-dimensional model of a nonlinearly elastic plate has solutions if the surface is rigid (in the sense of the uniqueness property \eqref{e7}). Thus, in order to prove an existence theorem for a plate under a specific set of data, it suffices to study its rigidity. 

In the next two sections, we provide two different types of sufficient conditions guaranteeing the rigidity of a plate: first we consider the case of boundary conditions imposed on the normal component of the displacement field (Section \ref{s4}), then we consider the case of boundary conditions imposed on the tangential components of the displacement field (Section \ref{s5}). Note that these sufficient conditions might not be necessary, but they cover most practical situations.

\section{A first type of rigidity theorem}
\label{s4}

In this section, we study the rigidity of plates subjected to boundary conditions of place on its transverse displacement fields $u_3:\omega\to{\capital{R}}$. We prove that such a plate is rigid by first using the implication  
\begin{equation*}
	\partial_\alpha u_\beta+\partial_\beta u_\alpha+\partial_\alpha u_3\partial_\beta u_3=0 \text{ in } \omega
	\Rightarrow
	\det (\partial_{\alpha\beta}u_3)=0 \text{ in } \omega
\end{equation*}
satisfied by all vector fields $(u_i)\in H^1(\omega)\times H^1(\omega)\times H^2(\omega)$, then by using boundary conditions on $u_3$ to deduce that $u_3$ must vanish in $\omega$. 

It is already known since Rabier \cite{rab} that the above implications hold for totally clamped plates, that is, when the admissible displacement fields $(u_i)\in H^1(\omega)\times H^1(\omega)\times H^2(\omega)$ satisfy the boundary conditions  
\begin{equation*}
	u_i=\partial_\nu u_3=0 \text{ \ on } \partial\omega.
\end{equation*}
In this section, we complete Rabier's result by showing that it suffices to assume that the admissible displacements only satisfy the boundary conditions  
\begin{equation*}
	u_3=0 \text{ \ on } \partial\omega
\end{equation*}
and the uniqueness property  
\begin{equation*}
	\partial_\alpha u_\beta+\partial_\beta u_\alpha =0 \textup{ in } \omega \textup{ imply } u_\alpha=0 \textup{ in } \omega.
\end{equation*}

The main result of this section is Theorem \ref{t3}. The proof of this theorem relies on two preliminary results, established below in Propositions \ref{p1} and \ref{p3}. Note that Proposition \ref{p3} generalizes a well-known rigidity result, stated in Proposition \ref{p2} for comparison.  

We begin by proving the following proposition, which will be needed in the proof of  Theorem \ref{t3}.

\begin{proposition}
	\label{p1}
	Let $\omega\subset {\capital{R}}^2$ be an open set. Then a vector field $(u_i)\in H^1(\omega)\times H^1(\omega)\times H^2(\omega)$ satisfies  
	\begin{equation*}
		\partial_\alpha u_\beta+\partial_\beta u_\alpha+\partial_\alpha u_3\partial_\beta u_3=0 \textup{ \ in } \omega
	\end{equation*}
	only if  
	\begin{equation*}
		\det (\partial_{\alpha\beta}u_3)=0 \textup{ \ in } \omega.
	\end{equation*}
\end{proposition}

\begin{proof}
	This lemma is a particular case of the nonlinear Saint-Venant compatibility conditions established in Ciarlet \& S. Mardare \cite[Theorem 4.1]{cs2013}, but can be established here in a simpler way by using the symmetry of higher order derivatives of distributions. More specifically, if $(u_i)\in H^1(\omega)\times H^1(\omega)\times H^2(\omega)$ satisfies  
	\begin{equation*}
		\partial_\alpha u_\beta+\partial_\beta u_\alpha+\partial_\alpha u_3\partial_\beta u_3=0 \textup{ \ in } \omega,
	\end{equation*}
	which is equivalent to 
	\begin{equation*}\aligned{}
		2\partial_1u_1&=-(\partial_1u_3)^2  \textup{ \ in } \omega, \\
		2\partial_2u_2&=-(\partial_2u_3)^2 \textup{ \ in } \omega, \\
		\partial_1u_2+\partial_2u_1&=-\partial_1u_3\partial_2u_3 \textup{ \ in } \omega,
		\endaligned\end{equation*}
	then
	\begin{equation*}\aligned{}
		0&=\partial_{11}(\partial_2u_2)+\partial_{22}(\partial_1u_1)-\partial_{12}(\partial_1u_2+\partial_2u_1)\\
		&=\partial_{12}(\partial_1u_3\partial_2u_3)-\frac12\partial_{11}((\partial_2u_3)^2)-\frac12\partial_{22}((\partial_1u_3)^2)\\
		&=\partial_{121}u_3\partial_2u_3+\partial_{21}u_3\partial_{12}u_3+\partial_{11}u_3\partial_{22}u_3+\partial_{1}u_3\partial_{122}u_3\\
		& -\partial_{112}u_3\partial_2u_3-\partial_{12}u_3\partial_{12}u_3-\partial_{221}u_3\partial_1u_3-\partial_{21}u_3\partial_{21}u_3
		\\
		&=\partial_{11}u_3\partial_{22}u_3-(\partial_{12}u_3)^2
		=\det (\partial_{\alpha\beta}u_3) \text{ \ in } \omega.
		\endaligned\end{equation*}
\end{proof}

The next result is well known (see, e.g., Rabier \cite{rab}), but we record it here with its proof for reader convenience.

\begin{proposition}
	\label{p2}
	Let $\omega\subset {\capital{R}}^2$ be a bounded open set. Then a function $u_3\in H^2_0(\omega)$ satisfies  
	\begin{equation*}
		\det (\partial_{\alpha\beta}u_3)=0 \textup{ \ in } \omega
	\end{equation*}
	only if  
	\begin{equation*}
		u_3=0 \textup{ \ in } \omega.
	\end{equation*}
\end{proposition}

\begin{proof}
	Given any function $f\in {\italic{D}}(\omega)$, a repeated use of the integration by parts formula
	shows that 
	\begin{equation*}\aligned{}
		\int_\omega |\nabla f|^2dy
		&=\frac12\int_\omega \Big[\partial_{22}((y_2)^2)(\partial_1f)^2+\partial_{11}((y_1)^2)(\partial_2f)^2\Big]dy\\
		&=\frac12\int_\omega \Big[(y_2)^2\partial_{22}((\partial_1f)^2)+(y_1)^2\partial_{11}((\partial_2f)^2)\Big]dy\\
		&=\int_\omega (y_2)^2\Big((\partial_{12}f)^2+\partial_1f\partial_{122}f\Big)dy\\
		& +\int_\omega (y_1)^2\Big((\partial_{12}f)^2+\partial_2f\partial_{112}f)\Big)dy\\
		&=\int_\omega (y_2)^2\Big(\partial_1(\partial_1f\partial_{22}f)-\det(\partial_{\alpha\beta}f)\Big)dy\\
		& +\int_\omega (y_1)^2\Big(\partial_2(\partial_2f\partial_{11}f)-\det(\partial_{\alpha\beta}f)\Big)dy.
		\endaligned\end{equation*}
	
	Since  
	\begin{equation*}
		(y_1)^2\partial_2(\partial_2f\partial_{11}f)=\partial_2((y_1)^2\partial_2f\partial_{11}f)
	\end{equation*}
	and  
	\begin{equation*}
		(y_2)^2\partial_1(\partial_1f\partial_{22}f)=\partial_1((y_2)^2\partial_1f\partial_{22}f),
	\end{equation*}
	we deduce that 
	\begin{equation*}\aligned{}
		\int_\omega |\nabla f|^2dy=-\int_\omega |y|^2 \det(\partial_{\alpha\beta}f) dy \text{ for all } f\in {\italic{D}}(\omega).
		\endaligned\end{equation*}
	
	Since $H^2_0(\omega)$ is the closure in $H^2(\omega)$ of ${\italic{D}}(\omega)$, given any function $u_3\in H^2_0(\omega)$, there exists a sequence $(f_n)_{n\in{\capital{N}}}$ such that $\|f_n-u_3\|_{H^2(\omega)}\to 0$ when $n\to \infty$. Then passing to the limit when $n\to\infty$ in the above equality with $f$ replaced by $f_n$ yields the equality
	\begin{equation*}\aligned{}
		\int_\omega |\nabla u_3|^2dy=-\int_\omega |y|^2 \det(\partial_{\alpha\beta}u_3) dy \text{ for all } u_3\in H^2_0(\omega).
		\endaligned\end{equation*}
	
	Therefore, if a function $u_3\in H^2_0(\omega)$ satisfies $\det (\partial_{\alpha\beta}u_3)=0$ in $\omega$, then $\nabla u_3=0$ in $\omega$, so $u_3=0$ in $\omega$ thanks to the assumption that $\omega$ is bounded. 
	
	\end{proof}

Proposition \ref{p2} shows that if a function $u_3\in H^2(\omega)$ satisfies the boundary condition 
\begin{equation*}
	u_3=\partial_\nu u_3=0 \textup{ \ on } \partial\omega,
\end{equation*}
in addition to the partial differential equation  
\begin{equation*}
	\det (\partial_{\alpha\beta}u_3)=0 \textup{ \ in } \omega,
\end{equation*}
where $\omega\subset{\capital{R}}^2$ is a Lipschitz domain, then necessarily $u_3=0$ in $\omega$. 

Proposition \ref{p3} below shows that the boundary condition
\begin{equation*}
	u_3=0 \textup{ \ on } \partial\omega,
\end{equation*}
instead of the above one in fact suffices to conclude that $u_3=0$ in $\omega$.

\begin{proposition}
	\label{p3}
	Let $\omega\subset {\capital{R}}^2$ be a bounded and connected open set with Lipschitz-continuous boundary. Then a function $u_3\in H^1_0(\omega)\cap H^2(\omega)$ satisfies  
	\begin{equation*}
		\det (\partial_{\alpha\beta}u_3)=0 \textup{ \ in } \omega
	\end{equation*}
	only if  
	\begin{equation*}
		u_3=0 \textup{ \ in } \omega.
	\end{equation*}
\end{proposition}

\begin{proof}
	Let $v\in \italic{C}^2(\overline\omega)$ be any function satisfying
	$$
	\sup_{y\in\partial\omega} v(y)<\sup_{x\in\overline\omega} v(x).
	$$
	Since 
	$$
	\aligned
	\sup_{x\in\overline\omega,\ y\in\partial\omega} (v(x)-v(y))
	\leq
	\sup_{x\in\overline\omega,\ y\in\partial\omega} (v(x)-v(y)-a\cdot (x-y))\\
	+|a|\sup_{x\in\overline\omega,\ y\in\partial\omega} |x-y| 
	\endaligned
	$$
	for every vector $a\in \capital{R}^2$, the previous inequality implies that
	\begin{equation}
		\label{e36}
		\sup_{x\in\overline\omega,\ y\in\partial\omega} (v(x)-v(y)-b\cdot (x-y))>0
	\end{equation}
	for all vectors $b\in \capital{R}^2$ that satisfy
	\begin{equation}
		\label{e37}
		|b|<R(v):=\frac1{\operatorname{diam}\omega} (\sup_{x\in\overline\omega} v(x)- \sup_{y\in\partial\omega} v(y)),
	\end{equation}
	where $\operatorname{diam}\omega$ denotes the diameter of $\omega$ (which is finite since $\omega$ is bounded).
	
	Inequality \eqref{e36} shows that, for each vector $b\in \capital{R}^2$ satisfying \eqref{e37}, the function $w_b\in \italic{C}^2(\overline\omega)$, defined by 
	$$
	w_b(x):=v(x)-b\cdot x \text{ for all } x\in \overline\omega,
	$$
	attained its maximum in $\omega$. Thus there exists $x_b\in\omega$ such that 
	$$
	w_b(x_b)=\sup_{z\in\omega} w_b(z).
	$$
	Since $\omega$ is an open set, this implies that $\nabla w_b(x_b)=0$, hence that 
	$$
	\nabla v(x_b)=b
	$$
	for all vectors $b\in \capital{R}^2$ satisfying \eqref{e37}. Therefore the open ball with radius $R(v)$ centered at the origin of $\capital{R}^2$ is contained in the image of $\omega$ by the mapping $\nabla v:\overline\omega\to \capital{R}^2$. Thus 
	$$
	\pi R(v)^2 \leq |\nabla v(\omega)|,
	$$
	where $|\nabla v(\omega)|$ denotes the area of the image of $\omega$ by the mapping $\nabla v$, on the one hand. 
	
	On the other hand, the area formula applied to the mapping $\nabla v\in \italic{C}^1(\overline\omega;\capital{R}^2)$ and on the open set $\omega\subset \overline\omega$ shows that (see. e.g., \cite[Theorem 1, p. 96]{eg})
	$$
	|\nabla v(\omega)|\leq \int_\omega |\det (\nabla(\nabla v))(z)| dz.
	$$
	Therefore
	$$
	\pi R(v)^2 \leq  \int_\omega |\det (\nabla(\nabla v))(z)| dz,
	$$
	or equivalently
	$$
	\sup_{x\in\overline\omega} v(x) - \sup_{y\in\partial\omega} v(y)
	\leq \frac{\operatorname{diam}\omega}{\sqrt\pi}\left\|\det(\partial_{\alpha\beta}v)\right\|_{L^1(\omega)}^{1/2}. 
	$$
	
	This inequality has been established above for all functions $v\in \italic{C}^2(\overline\omega)$ satisfying
	$$
	\sup_{y\in\partial\omega} v(y)<\sup_{x\in\overline\omega} v(x), 
	$$
	but it obviously holds also for all functions $v\in \italic{C}^2(\overline\omega)$ satisfying
	$$
	\sup_{y\in\partial\omega} v(y)=\sup_{x\in\overline\omega} v(x). 
	$$
	Hence
	$$
	\sup_{x\in\overline\omega} v(x) \leq \sup_{y\in\partial\omega} |v(y)|
	+ \frac{\operatorname{diam}\omega}{\sqrt\pi}\left\|\det(\partial_{\alpha\beta}v)\right\|_{L^1(\omega)}^{1/2}
	$$
	for all functions $v\in \italic{C}^2(\overline\omega)$. 
	
	Combining the above inequality with a similar one for the function $(-v)$ instead of $v$ shows that 
	$$
	\sup_{x\in\overline\omega} |v(x)| \leq \sup_{y\in\partial\omega} |v(y)|
	+ \frac{\operatorname{diam}\omega}{\sqrt\pi}\left\|\det(\partial_{\alpha\beta}v)\right\|_{L^1(\omega)}^{1/2}
	$$
	for all functions $v\in \italic{C}^2(\overline\omega)$. 
	
	Since the space $\italic{C}^2(\overline\omega)$ is dense in the space $H^2(\omega)$ and the embedding $H^2(\omega)\subset \italic{C}^0(\overline\omega)$ is continuous (thanks to the assumptions on the domain $\omega\subset \capital{R}^2$; see, e.g., \cite{ada}), the above inequality hold in fact for all $v\in H^2(\omega)$. In particular then, 
	$$
	\sup_{x\in\overline\omega} |u(x)| 
	\leq \frac{\operatorname{diam}\omega}{\sqrt\pi}\left\|\det(\partial_{\alpha\beta}u)\right\|_{L^1(\omega)}^{1/2}
	$$
	for all functions $u\in H^2(\omega)\cap H^1_0(\omega)$. 
	
	Consequently, if a function $u\in H^1_0(\omega)\cap H^2(\omega)$ satisfies  
	\begin{equation*}
		\det (\partial_{\alpha\beta}u)=0 \textup{ \ in } \omega, 
	\end{equation*}
	then $u$ vanishes in $\omega$. 
\end{proof}

A consequence of Proposition \ref{p3}, combined with Proposition \ref{p1}, is the following rigidity theorem for plates. Note that other different rigidity theorems will be proved in the next section. Remember that the tensor field ${\boldletter{E}}({\boldletter{u}})=(E_{\alpha\beta}({\boldletter{u}})):\omega\to{\capital{R}}^{2\times 2}$ associated with a vector field ${\boldletter{u}}:=(u_i):\omega\to{\capital{R}}^3$ is defined by  
\begin{equation*}
	E_{\alpha\beta}({\boldletter{u}}):=\frac12(\partial_\alpha u_\beta+\partial_\beta u_\alpha+\partial_\alpha u_3\partial_\beta u_3).
\end{equation*}
Note that the conclusion of the next theorem means that the set $V(\omega)$ satisfies all the assumptions of Theorem \ref{t2}, thus ensuring the existence of minimizers for the two-dimensional model of nonlinearly elastic plates.

\begin{thm}
	\label{t3}
	Let $\omega\subset {\capital{R}}^2$ be a bounded and connected open set with Lipschitz-continuous boundary. Assume that the set $V(\omega)=V_H(\omega)\times V_3(\omega)$, where $V_H(\omega)\subset H^1(\omega)\times H^1(\omega)$ and $V_3(\omega)\subset H^2(\omega)$, satisfy the following two conditions:  
	\begin{equation*}
		(u_\alpha)\in V_H(\omega) \textup{ and } \partial_\alpha u_\beta+\partial_\beta u_\alpha =0 \textup{ in } \omega \textup{ imply } u_\alpha=0 \textup{ in } \omega,
	\end{equation*}
	and  
	\begin{equation*}
		u_3\in V_3(\omega) \textup{ implies } u_3=0 \textup{ on } \partial\omega.\end{equation*}
	Then the set $V(\omega)$ satisfies the following rigidity property: 
	\begin{equation}\label{e38}\aligned{} 
		u_3 \in V_3(\omega) \textup{ and } \partial_{\alpha\beta}u_3=0 \textup{ in } \omega \textup{ imply } u_3=0 \textup{ in } \omega,
		\endaligned\end{equation}
	and 
	\begin{equation}\label{e39}\aligned{} 
		{\boldletter{u}} \in V(\omega) \textup{ and } {\boldletter{E}}({\boldletter{u}})={\boldletter{0}} \textup{ in } \omega \textup{ imply } {\boldletter{u}}={\boldletter{0}} \textup{ in } \omega.
		\endaligned\end{equation}
\end{thm}

\begin{proof}
	First, let $u_3 \in V_3(\omega)$ satisfy $\partial_{\alpha\beta}u_3=0$ in $\omega$. Since $\omega$ is connected by assumption, this implies that $u_3$ is an affine function that vanishes on the boundary of $\omega$. Since $\omega$ is also bounded by assumption, this implies that $u_3=0$ in $\omega$.
	
	Second, let ${\boldletter{u}}=(u_i) \in V(\omega)$ satisfy ${\boldletter{E}}({\boldletter{u}})={\boldletter{0}}$ in $\omega$. Then Proposition \ref{p1} implies that  
	\begin{equation*}
		\det (\partial_{\alpha\beta}u_3)=0 \textup{ \ in } \omega.
	\end{equation*}
	Besides, ${\boldletter{u}}=(u_i) \in V(\omega)$ implies that $u_3\in V_3(\omega)$, which in turn implies that $u_3\in H^2(\omega)\cap H^1_0(\omega)$ by the assumptions of the theorem. This means that $u_3$ satisfies all the assumptions of Proposition \ref{p3}, so  
	\begin{equation*}
		u_3=0 \text{ in } \omega.
	\end{equation*}
	
	Then we infer from the relations ${\boldletter{u}}\in V(\omega)$ and ${\boldletter{E}}({\boldletter{u}})={\boldletter{0}}$ in $\omega$, and from  the definition of the components $E_{\alpha\beta}({\boldletter{u}})$ of  ${\boldletter{E}}({\boldletter{u}})$, that  
	\begin{equation*}
		(u_\alpha)\in V_H(\omega) \text{ and } \partial_\alpha u_\beta+\partial_\beta u_\alpha =0 \textup{ in } \omega.  
	\end{equation*}
	Then the assumptions of the theorem imply that  
	\begin{equation*}
		u_\alpha=0 \textup{ in } \omega.
	\end{equation*}
	Thus ${\boldletter{u}}={\boldletter{0}}$ in $\omega$, which completes the proof of the theorem.
	\end{proof}

An immediate consequence of Theorems \ref{t2} and \ref{t3} is the following existence result for the Kirchhoff-Love nonlinear model of plates:

\begin{cor}
	\label{c1}
	Let $\omega\subset {\capital{R}}^2$ be a bounded and connected open set with Lipschitz-continuous boundary. Let $V_3(\omega)$ be any closed subspace of $H^2(\omega)\cap H^1_0(\omega)$ and let $V_H(\omega)$ be any closed subspace of $H^1(\omega)\times H^1(\omega)$ such that  
	\begin{equation*}
		V_H(\omega)\cap \operatorname{Rig}(\omega)=\{{\boldletter{0}}\},
	\end{equation*}
	where $\operatorname{Rig}(\omega)$ is the three-dimensional space formed by the functions  
	\begin{equation*}
		(y_1,y_2)\in\omega \mapsto (a,b)+c(y_2, -y_1)\in{\capital{R}}^2,
	\end{equation*}
	where $a,b,c\in{\capital{R}}$.
	
	Then the functional $J$ defined by \eqref{e1}-\eqref{e6} has a minimizer over the space $V(\omega)=V_H(\omega)\times V_3(\omega)$. 
\end{cor}

\section{A second type of rigidity theorems}
\label{s5}

In this section, we study the rigidity of plates subjected to boundary conditions of place on its tangential displacement fields $(u_\alpha):\omega\to{\capital{R}}^2$. We first prove that a plate is rigid if $(u_\alpha)$ vanish on the entire boundary of $\omega$ as a simple consequence of the integration by parts formula. We then substantially improve this rigidity result by using the implication  
\begin{equation*}
	\partial_\alpha u_\beta+\partial_\beta u_\alpha+\partial_\alpha u_3\partial_\beta u_3=0 \text{ in }\omega \Rightarrow 
	\left\{
	\aligned
	\partial_1 u_1\leq 0  \text{ in }\omega, \\
	\partial_2 u_2\leq 0  \text{ in }\omega,
	\endaligned
	\right.
\end{equation*}
satisfied by all vector fields $(u_i)\in H^1(\omega)\times H^1(\omega)\times H^2(\omega)$, then by using boundary conditions on $u_\alpha$, not necessarily on the entire boundary of $\omega$,  to deduce that $u_\alpha$ must vanish in $\omega$. 

We begin with the simpler case where the admissible displacements fields ${\boldletter{u}}=(u_i):\omega\to{\capital{R}}^3$ satisfy the boundary conditions $u_\alpha = 0$ on $\partial\omega$.

\begin{thm}
	\label{t4}
	Let $\omega\subset {\capital{R}}^2$ be a bounded and connected open set with Lipschitz-continuous boundary. Assume that the set $V(\omega)=V_H(\omega)\times V_3(\omega)$, where $V_H(\omega)\subset H^1(\omega)\times H^1(\omega)$ and $V_3(\omega)\subset H^2(\omega)$, satisfy the following two conditions: 
	\begin{equation}\label{e40}\aligned{} 
		u_3 \in V_3(\omega) \textup{ and } \partial_{\alpha}u_3=0 \textup{ in } \omega \textup{ imply } u_3=0 \textup{ in } \omega,
		\endaligned\end{equation}
	and 
	\begin{equation}\label{e41}\aligned{} 
		(u_\alpha)\in V_H(\omega) \textup{ implies } u_\alpha=0 \textup{ on } \partial\omega. 
		\endaligned\end{equation}
	
	Then the set $V(\omega)$ satisfies the following rigidity property:
	\begin{equation}\label{e42}\aligned{} 
		{\boldletter{u}} \in V(\omega) \text{ and } {\boldletter{E}}({\boldletter{u}})={\boldletter{0}} \text{ in } \omega \text{ implies } {\boldletter{u}}={\boldletter{0}} \text{ in } \omega.
		\endaligned\end{equation}
\end{thm}

\begin{proof}
	Let ${\boldletter{u}}\in V(\omega)$ satisfy ${\boldletter{E}}({\boldletter{u}})={\boldletter{0}}$ in $\omega$. Thanks to the boundary conditions \eqref{e41} satisfied by $u_\alpha$, integrating this relation over $\omega$ and using the integration by parts formula yields  
	\begin{equation*}
		\int_\omega \partial_\alpha u_3\partial_\beta u_3 dy=0.
	\end{equation*}
	In particular then,  
	\begin{equation*}
		\int_\omega (\partial_\alpha u_3)^2dy=0,
	\end{equation*}
	which combined with assumption \eqref{e40} satisfied by $u_3$ implies that
	\begin{equation*}
		u_3 = 0 \textrm{ in } \omega.
	\end{equation*}
	
	Consequently, the relation ${\boldletter{E}}({\boldletter{u}}) = {\boldletter{0}}$ in $\omega$ implies that  
	\begin{equation*}
		\partial_\beta u_\alpha + \partial_\alpha u_\beta = 0 \textrm{ in } \omega,
	\end{equation*}
	which in turn implies that  
	\begin{equation*}
		\partial_{\alpha\beta}u_\sigma=0  \textrm{ in } \omega.  
	\end{equation*}
	The set $\omega$ being connected, the functions $u_\sigma$ must be affine. Since in addition they vanish on the entire boundary of a bounded set, so necessarily at three non-collinear points, they must vanish in $\omega$:  
	\begin{equation*}
		u_\alpha = 0 \textrm{ in } \omega.
	\end{equation*}
	This completes the proof of the theorem. 
	\end{proof}

An immediate consequence of Theorem \ref{t4} combined with Theorem \ref{t2} is that Rabier's existence result in \cite{rab} holds and the weaker assumptions of Theorem \ref{t4} instead of Rabier's assumptions that $u_\alpha=u_3=\partial_\nu u_3 = 0$ on $\partial\omega$. Note also the simpler proof given here to the rigidity result of Theorem \ref{t4} compared with the one in \cite{rab}. More specifically, we have the following existence result that supersedes Rabier's one in \cite{rab}: 

\begin{cor}
	\label{c2}
	Let $\omega\subset {\capital{R}}^2$ be a bounded and connected open set with Lipschitz-continuous boundary. 
	
	Then the functional $J$ defined by \eqref{e1}-\eqref{e6} has a minimizer in the set $V(\omega)=V_H(\omega)\times V_3(\omega)$, where $V_H(\omega)$ is any closed subspace of the space  
	\begin{equation*}
		\{(u_\alpha)\in H^1(\omega)\times H^1(\omega);\ u_\alpha=0 \textup{ on } \partial\omega\}
	\end{equation*}
	and $V_3(\omega)$ is any closed subspace of the space $H^2(\omega)$ whose intersection with the space of affine functions over $\omega$ reduces to $\{0\}$. 
\end{cor}

In the remaining part of this section we further weaken the assumptions of Rabier's existence theorem by showing that $u_\alpha$ need not vanish on the entire boundary of $\omega$, but only on a part of it, for the conclusion of Corollary \ref{c2} to hold. 

However, the portion of the boundary of $\omega$ where $u_\alpha$ vanish cannot be too small, as shown by the following example. 
Let $\omega := (0,1) \times (0,1)$ and $\gamma_0:=\{0\}\times (0,1)$. Let $f\in H^2((0,1))$ be any non-zero function that satisfies $f(0)=f'(0)=0$. Then the vector field ${\boldletter{u}}=(u_i):\omega\to {\capital{R}}^3$ defined by 
\begin{equation*}\aligned{}
	u_2(y)&:=0 &&  \text{ for all } y\in\omega,\\
	u_3(y)&:=f(y_1) &&  \text{ for all } y=(y_\alpha)\in\omega,\\
	u_1(y)&:=-\frac12\int_0^{y_1}(f'(t))^2dt && \text{ for all } y=(y_\alpha)\in\omega,
	\endaligned\end{equation*}
belongs to the space $V(\omega)=V_H(\omega)\times V_3(\omega)$, where 
\begin{equation*}\aligned{}
	V_H(\omega)&:=\{(u_\alpha)\in H^1(\omega)\times H^1(\omega); \ u_\alpha=0 \text{ on } \gamma_0\},\\
	V_3(\omega)&:=\{u_3\in H^2(\omega);\ u_3=\partial_\nu u_3=0 \text{ on } \gamma_0\},
	\endaligned\end{equation*}
and satisfies  
\begin{equation*}
	{\boldletter{E}}({\boldletter{u}}) = {\boldletter{0}} \textrm{ in } \omega,
\end{equation*}
but ${\boldletter{u}}$ does not vanish in $\omega$. 

To begin with, we study the case of a particular class of domains $\omega$, defined as follows:

\begin{df}
	\label{d1}
	A set $\omega \subset {\capital{R}}^2$ is rectangle-like with respect to the vertical axis if 
	\begin{equation}\label{e43}\aligned{} 
		\omega=\{(y_1,y_2)\in {\capital{R}}^2; \ a<y_1<b, \ f(y_1) <y_2<g(y_1)\}
		\endaligned\end{equation}
	for some real numbers $a<b$ and some Lipschitz-continuous functions  
	\begin{equation*}
		f,g:[a,b]\to{\capital{R}}
	\end{equation*}
	satisfying $f(t)<g(t)$ for all $t\in [a,b]$. 
	
\end{df}

\begin{rem}
	\label{r1}
	A similar definition holds for sets that are rectangle-like with respect to the horizontal axis, in which case a theorem similar to Theorem \ref{t5} can be established by using the same argument.
	\end{rem}

The interest of Definition \ref{d1} for the purpose of this paper is that the rigidity result Theorem \ref{t4} can be improved for domains $\omega$ that are rectangle-like with respect to one cartesian axis. Note that the assumption \eqref{e44} in the next theorem means that the space $V_3(\omega)$ does not contain non-zero constant functions. 

More specifically, we have the following rigidity result: 

\begin{thm}
	\label{t5}
	Let $\omega\subset {\capital{R}}^2$ be a rectangle-like set with respect to the vertical axis. Let $\gamma_f\subset \partial\omega$ and $\gamma_g\subset\partial\omega$ denote the graphs of the functions $f$ and $g$ appearing in the definition \eqref{e43} of the set $\omega$. Assume that the set $V(\omega)=V_H(\omega)\times V_3(\omega)$, where $V_H(\omega)\subset H^1(\omega)\times H^1(\omega)$ and $V_3(\omega)\subset H^2(\omega)$, satisfy the following two conditions: 
	\begin{equation}\label{e44}\aligned{} 
		u_3 \in V_3(\omega) \textup{ and } \partial_{\alpha}u_3=0 \textup{ in } \omega \textup{ imply } u_3=0 \textup{ in } \omega,
		\endaligned\end{equation}
	and 
	\begin{equation}\label{e45}\aligned{} 
		(u_\alpha)\in V_H(\omega) \textup{ implies } u_\alpha=0 \textup{ on } \gamma_f\cup\gamma_g. 
		\endaligned\end{equation}
	
	Then the set $V(\omega)$ satisfies the following rigidity property:
	\begin{equation}\label{e46}\aligned{} 
		{\boldletter{u}} \in V(\omega) \text{ and } {\boldletter{E}}({\boldletter{u}})={\boldletter{0}} \text{ in } \omega \text{ implies } {\boldletter{u}}={\boldletter{0}} \text{ in } \omega.
		\endaligned\end{equation}
\end{thm}

\begin{proof} 
	If a vector field ${\boldletter{u}} \in V(\omega)$ satisfies ${\boldletter{E}}({\boldletter{u}})={\boldletter{0}}$ in $\omega$, it satisfies in particular 
	\begin{equation}\label{e47}\aligned{} 
		\partial_2 u_2=-\frac12(\partial_2u_3)^2\leq 0 \text{ \ in } \omega,
		\endaligned\end{equation}
	on the one hand.
	
	On the other hand, the definition \eqref{e43} of the set $\omega$ implies that, for each $c \in (a,b)$, the intersection of the straight line $y_1=c$ with ${\overline\omega}$ is a line segment  
	\begin{equation*}
		\omega_c:=\{(c,y_2); f(c)\leq y_2\leq g(c)\},
	\end{equation*}
	whose endpoints belong to the set  $\gamma_f\cup\gamma_g$. Since $u_2\in H^1(\omega)$, a classical theorem about Sobolev spaces (see, e.g., Maz'ya \cite[Theorem 1, \S 1.1.3]{maz}) shows that, for almost all $c\in (a,b)$, the trace of $u_2$ on ${\omega_c}$ is absolutely continuous and its derivative is equal to $\partial_2u_2$ almost everywhere on $\omega_c$. 
	
	Then we infer from \eqref{e47} and from assumptions \eqref{e45} that
	\begin{equation*}
		u_2|_{\omega_c} = 0
	\end{equation*}
	for almost all $c\in (a,b)$, which in turn implies that 
	\begin{equation}\label{e48}\aligned{} 
		u_2 = 0 \text{ in } \omega, 
		\endaligned\end{equation}
	then that 
	\begin{equation}\label{e49}\aligned{} 
		(\partial_2 u_3)^2=-2\partial_2u_2 = 0 \textrm{ in } \omega.
		\endaligned\end{equation}
	
	Next we infer from the assumption ${\boldletter{E}}({\boldletter{u}})={\boldletter{0}}$ in $\omega$ that  
	\begin{equation*}
		\partial_2 u_1 + \partial_1 u_2 + \partial_1 u_3 \partial_2 u_3 =0 \textrm{ in } \omega,
	\end{equation*}
	which combined with \eqref{e48} and \eqref{e49} shows that  
	\begin{equation*}
		\partial_2 u_1 = 0 \textrm{ a.e in } \omega.
	\end{equation*}
	This means that $u_1$ is constant with respect to the variable $y_2$, so it must vanish since its trace on $\gamma_f$ vanishes. So  
	\begin{equation*}
		u_1= 0 \text{ in } \omega,  
	\end{equation*}
	which combined with the relation (remember that ${\boldletter{E}}({\boldletter{u}})={\boldletter{0}}$ in $\omega$)  
	\begin{equation*}
		E_{11}({\boldletter{u}}):=\partial_1u_1+\frac12 (\partial_1u_3)^2=0 \text{ in } \omega,
	\end{equation*}
	next implies that  
	\begin{equation*}
		\partial_1u_3= 0 \text{ in } \omega.
	\end{equation*}
	Since we also have $\partial_2u_3= 0$ in $\omega$ (see relation \eqref{e49} above), assumption  \eqref{e44} implies that  
	\begin{equation*}
		u_3= 0 \text{ in } \omega.
	\end{equation*}
	Since we already proved that $u_1=0$ and $u_2=0$ in $\omega$, this completes the proof of the theorem.
\end{proof}

Combined with Theorem \ref{t2}, an  immediate consequence of Theorem \ref{t5} is the following existence result for partially clamped plates, this first in the literature to hold for this kind of plates without any smallness assumption on the applied forces.

\begin{cor}
	\label{c3}
	Let $\omega\subset {\capital{R}}^2$ be a rectangle-like set with respect to the vertical axis. Let $\gamma_f\subset \partial\omega$ and $\gamma_g\subset\partial\omega$ denote the graphs of the functions $f$ and $g$ appearing in the definition \eqref{e43} of the set $\omega$. 
	
	Then the functional $J$ defined by \eqref{e1}-\eqref{e6} has a minimizer in the set $V(\omega)=V_H(\omega)\times V_3(\omega)$, where  
	\begin{equation*}
		V_H(\omega):=\big\{(u_\alpha)\in H^1(\omega)\times H^1(\omega); \ u_\alpha=0 \textup{ on } \gamma_f\cup\gamma_g\big\}
	\end{equation*}
	and $V_3(\omega)$ is any closed subspace of $H^2(\omega)$ whose intersection with the space of affine functions over $\omega$ reduces to $\{0\}$. 
\end{cor}

We conclude this section by generalizing the rigidity result of Theorem \ref{t5}. Note that the assumption \eqref{e50} in the next theorem means that the space $V_3(\omega)$ does not contain non-zero constant functions, while the assumption \eqref{e51} means that the elements of the space $V_H(\omega)$ vanish at all points of the boundary where the tangent exists and is not parallel to the given vector $(e_\beta)\in{\capital{R}}^2$. The typical example is the rectangle-like set of Theorem \ref{t5}, for which $(e_\beta)=(0,1)\in{\capital{R}}^2$. 

\begin{thm}
	\label{t6}
	Let $\omega\subset {\capital{R}}^2$ be a bounded and connected open subset of ${\capital{R}}^2$ whose boundary is Lipschitz-continuous. Assume that the set $V(\omega)=V_H(\omega)\times V_3(\omega)$, where $V_H(\omega)\subset H^1(\omega)\times H^1(\omega)$ and $V_3(\omega)\subset H^2(\omega)$, satisfy the following two conditions: 
	\begin{equation}\label{e50}\aligned{} 
		u_3 \in V_3(\omega) \textup{ and } \partial_{\alpha}u_3=0 \textup{ in } \omega \textup{ imply } u_3=0 \textup{ in } \omega,
		\endaligned\end{equation}
	and there exists a non-zero vector $(e_\beta)\in{\capital{R}}^2$ such that 
	\begin{equation}\label{e51}\aligned{} 
		(u_\alpha)\in V_H(\omega) \textup{ implies } u_\alpha(\nu_\beta e_\beta)=0 \textup{ on } \partial\omega. 
		\endaligned\end{equation}
	
	Then the set $V(\omega)$ satisfies the following rigidity property:
	\begin{equation}\label{e52}\aligned{} 
		{\boldletter{u}} \in V(\omega) \text{ and } {\boldletter{E}}({\boldletter{u}})={\boldletter{0}} \text{ in } \omega \text{ implies } {\boldletter{u}}={\boldletter{0}} \text{ in } \omega.
		\endaligned\end{equation}
\end{thm}

\begin{proof} 
	Without loss of generality we may assume that $(e_\beta)=(1,0)\in{\capital{R}}^2$, in which case assumption \eqref{e51} becomes 
	\begin{equation}\label{e53}\aligned{} 
		(u_\alpha)\in V_H(\omega) \textup{ implies } u_\alpha\nu_1=0 \textup{ on } \partial\omega. 
		\endaligned\end{equation}
	Let ${\boldletter{u}} \in V(\omega)$ be a vector field that satisfies the matrix equation ${\boldletter{E}}({\boldletter{u}})={\boldletter{0}}$ in $\omega$. Then 
	\begin{equation*}\aligned{}
		\partial_1 u_1+\frac12(\partial_1u_3)^2&= 0 \text{ \ in } \omega,\\
		\partial_1 u_2+\partial_2u_1+ \partial_1u_3\partial_2 u_3 &= 0 \text{ \ in } \omega,\\
		\partial_2 u_2+\frac12(\partial_2 u_3)^2 &= 0 \text{ \ in } \omega.
		\endaligned\end{equation*}
	Integrating the first equation over $\omega$ and using the integration by parts formula in conjunction with assumption \eqref{e53} satisfied by the trace of $u_1$ on the boundary $\gamma:=\partial\omega$ of $\omega$ implies that 
	\begin{equation*}\aligned{}
		\int_\omega(\partial_1u_3)^2dy=-2\int_{\partial\omega}u_1\nu_1 d\gamma=0, 
		\endaligned\end{equation*}
	then that 
	\begin{equation}\label{e54}\aligned{} 
		\partial_1 u_3&= 0 \text{ \ in } \omega.
		\endaligned\end{equation}
	
	Therefore, the previous system becomes
	\begin{equation}\label{e55}\aligned{} 
		\partial_1 u_1&= 0 \text{ \ in } \omega,\\
		\partial_1 u_2+\partial_2u_1&= 0 \text{ \ in } \omega,\\
		\partial_2 u_2+\frac12(\partial_2 u_3)^2 &= 0 \text{ \ in } \omega.
		\endaligned\end{equation}
	
	Using that $\partial_1 u_1= 0$ in $\omega$ and $u_1\nu_1=0$ on $\partial\omega$ (see relations \eqref{e53} and \eqref{e55}) in the integration by parts formula below shows that 
	\begin{equation}\label{e56}\aligned{} 
		\int_\omega (u_1(y))^2 dy
		=\int_\omega u_1(y)\partial_1\big(y_1u_1(y)\big)dy
		=\int_{\partial\omega}(u_1\nu_1)(y)y_1u_1(y) d\gamma=0. 
		\endaligned\end{equation}
	Hence
	\begin{equation}\label{e57}\aligned{} 
		u_1&= 0 \text{ \ in } \omega.
		\endaligned\end{equation}
	
	Next, using that $u_1$ vanishes in the second relation of system \eqref{e55} shows that 
	\begin{equation*}
		\partial_1u_2=0 \text{ in } \omega.
	\end{equation*}
	Since in addition $u_2\nu_1=0$ on $\partial\omega$ (see relation \eqref{e53}),  a similar argument as in \eqref{e56} above shows that 
	\begin{equation}\label{e58}\aligned{} 
		\int_\omega (u_2(y))^2 dy=\int_\omega u_2(y)\partial_1\big(y_1u_2(y)\big)dy=\int_{\partial\omega}(u_2\nu_1)(y)y_1u_2(y)d\gamma=0. 
		\endaligned\end{equation}
	Hence
	\begin{equation}\label{e59}\aligned{} 
		u_2&= 0 \text{ \ in } \omega.
		\endaligned\end{equation}
	
	Finally, the last equation of system \eqref{e55} implies that 
	\begin{equation}\label{e60}\aligned{} 
		\partial_2u_3&= 0 \text{ \ in } \omega,
		\endaligned\end{equation}
	which combined with relation \eqref{e54} and with assumption \eqref{e50} of the theorem implies that  
	\begin{equation*}
		u_3=0 \text{ \ in } \omega.
	\end{equation*}
	
	Since we already proved that $u_1=0$ and $u_2=0$ in $\omega$, this completes the proof of the theorem.
	\end{proof}

We conclude this section and paper by stating the existence result for the Kirchhoff-Love nonlinear plate model obtained by using the rigidity result of Theorem \ref{t6} to replace the assumptions of the existence result of Theorem \ref{t2}: 

\begin{cor}
	\label{c4}
	Let $\omega\subset {\capital{R}}^2$ be a bounded and connected open subset of ${\capital{R}}^2$ whose boundary is Lipschitz-continuous, let $(\nu_\alpha)\in L^\infty(\partial\omega;{\capital{R}}^2)$ denote unit outer normal vector field to the boundary of $\omega$, let $(e_\alpha)\in{\capital{R}}^2$ be any non-zero vector, and let   
	\begin{equation*}
		\gamma_0:=\{y\in\partial\omega; \ \nu_\alpha (y)e_\alpha\neq 0\}.
	\end{equation*}

	Then the functional $J$ defined by \eqref{e1}-\eqref{e6} has a minimizer in the set $V(\omega)=V_H(\omega)\times V_3(\omega)$, where  
	\begin{equation*}
		V_H(\omega):=\big\{(u_\alpha)\in H^1(\omega)\times H^1(\omega); \ u_\alpha=0 \textup{ on } \gamma_0\big\}
	\end{equation*}
	and $V_3(\omega)$ is any closed subspace of $H^2(\omega)$ whose intersection with the space of affine functions over $\omega$ reduces to $\{0\}$. 
\end{cor}

\section*{Acknowledgements} 

This work has been completed while the second author was visiting the Institute of Advanced Studies in Mathematics of Harbin Institute of Technology in China, whose hospitality and support is hereby gratefully acknowledged. 

The first author is supported by the grant PRIMUS/24/SCI/020 of Charles University and within the frame of the project Ferroic Multifunctionalities (FerrMion) [project No. CZ.02.01.01/00/22 008/0004591], within the Operational Programme Johannes Amos Comenius co-funded by the European Union (JS).


\newcommand \auth{} 
\newcommand \jour {} 
\newcommand \book {}

\end{document}